\documentclass[10pt]{article} 

\usepackage{amsmath,amsfonts,bm}

\def\eqref#1{equation~\ref{#1}}

\def\1{\bm{1}}

\DeclareMathAlphabet{\mathsfit}{\encodingdefault}{\sfdefault}{m}{sl}
\SetMathAlphabet{\mathsfit}{bold}{\encodingdefault}{\sfdefault}{bx}{n}

\def\gD{{\mathcal{D}}}

\def\gL{{\mathcal{L}}}

\DeclareMathOperator*{\argmin}{arg\,min}

\usepackage[utf8]{inputenc} 
\usepackage[T1]{fontenc}    
\usepackage{url}            
\usepackage{booktabs}       
\usepackage{amsfonts}       
\usepackage{nicefrac}       
\usepackage{microtype}      
\usepackage{fullpage}

\usepackage{hyperref}
\usepackage{cleveref}

\usepackage{algorithm}
\usepackage{algorithmic}
\usepackage{amsthm} 
\usepackage{natbib, comment}

\newtheorem{theorem}{Theorem}
 
\newtheorem{lemma}{Lemma} 
\newtheorem{assumption}{Assumption}
\newtheorem{corollary}{Corollary} 
\newtheorem{definition}{Definition}

\crefname{assumption}{assumption}{assumptions}
\theoremstyle{remark}

\allowdisplaybreaks

\usepackage{colortbl}
\usepackage{multirow}

\usepackage{threeparttable}
\usepackage[table]{xcolor}

\usepackage{hhline}
\usepackage{makecell}

\usepackage{subfigure}
\usepackage{graphicx}
\usepackage{mathtools}

\newcommand{\cS}{\mathcal{S}}

\title{A Primal-Dual  Approach to Bilevel Optimization with Multiple Inner Minima}

\author{Daouda Sow$^{1}$, Kaiyi Ji$^2$, Ziwei Guan$^{1}$, Yingbin Liang$^{1}$\\
$^{1}$Department of ECE, The Ohio State University\\ 
$^{2}$Department of EECS, University of Michigan, Ann Arbor\\
\texttt{sow.53@osu.edu, kaiyiji@umich.edu, liang889@osu.edu} \\ 
}

\begin{document}

\maketitle

\begin{abstract}
\noindent Bilevel optimization has found extensive applications in modern machine learning problems such as hyperparameter optimization, neural architecture search, meta-learning, etc. While bilevel problems with a unique inner minimal point (e.g., where the inner function is strongly convex) are well understood, such a problem with multiple inner minimal points remains to be challenging and open. Existing algorithms designed for such a problem were applicable to restricted situations and do not come with a full guarantee of convergence. In this paper, we adopt a reformulation of bilevel optimization to constrained optimization, and solve the problem via a primal-dual bilevel optimization (PDBO) algorithm. PDBO not only addresses the multiple inner minima challenge, but also features fully first-order efficiency without involving second-order Hessian and Jacobian computations, as opposed to most existing gradient-based bilevel algorithms. We further characterize the convergence rate of PDBO, which serves as the first known non-asymptotic convergence guarantee for bilevel optimization with multiple inner minima. Our experiments demonstrate desired performance of the proposed approach.

\end{abstract}

\section{Introduction}
	Bilevel optimization has received extensive attention recently due to its applications in a variety of modern machine learning problems. Typically, parameters handled by bilevel optimization are divided into two different types such as meta and base learners in few-shot meta-learning~\cite{bertinetto2018meta,rajeswaran2019meta}, hyperparameters and model parameters training in automated hyperparameter tuning~\cite{franceschi2018bilevel,shaban2019truncated}, actors and critics in reinforcement learning~\cite{konda2000actor,hong2020two}, and model architectures and weights in neural architecture search~\cite{liu2018darts}. 
	
	Mathematically, bilevel optimization captures intrinsic hierarchical structures in those machine learning models, and can be formulated into the following two-level problem:
	\begin{align}\label{eq:prob}
		\min_{x \in \mathcal{X}, y \in \mathcal{S}_x} f(x, y)
		\quad\mbox{with} \quad \mathcal{S}_x = \argmin_{y \in \mathcal{Y}} g(x,y),
	\end{align}
	where $f(x, y)$ and $g(x, y)$, the outer- and inner-level objective functions, are continuously differentiable, and the supports $\mathcal{X}\subseteq \mathbb{R}^p$ and $\mathcal{Y}\subseteq \mathbb{R}^d$ are convex, closed and bounded. For a fixed $x\in\mathcal{X}$, $\mathcal{S}_x$ is the set of all $y\in\mathcal{Y}$ that yields the minimal value of $g(x, \cdot)$.
	
	A broad collection of approaches have been proposed to solve the bilevel problem in~\cref{eq:prob}. Among them, gradient based algorithms have shown great effectiveness and efficiency in various deep learning applications, which include approximated implicit differentiation (AID) based methods~\citep{domke2012generic,pedregosa2016hyperparameter,gould2016differentiating,liao2018reviving,lorraine2020optimizing,ji2021bo} and iterative differentiation (ITD) (or dynamic system) based methods~\citep{maclaurin2015gradient,franceschi2017forward,shaban2019truncated,grazzi2020iteration,liu2020generic,liu2021value}. 
	Many stochastic bilevel algorithms have been proposed recently via stochastic gradients~\cite{ghadimi2018approximation,hong2020two,ji2021bo}, and variance reduction~\cite{yang2021provably} and momentum~\cite{chen2021single,khanduri2021near,guo2021randomized}.

	Most of these studies rely on the simplification that for each outer variable $x$, the inner-level problem has a {\bf single} global minimal point. The studies for a more challenging scenario with {\bf multiple} inner-level solutions (i.e., $\mathcal{S}_x$ has multiple elements) are rather limited. In fact, a counter example has been provided in \cite{liu2020generic} to illustrate that simply applying algorithms designed for the  single inner minima case will fail to optimize bilevel problems with multiple inner minima. Thus, bilevel problems with multiple inner minima deserve serious efforts of exploration. Recent studies \cite{liu2020generic,li2020improved} proposed a gradient aggregation method and another study \cite{liu2021value} proposed a value-function-based method from a constrained optimization view to address the issue of multiple inner minima. However, all of these approaches take a {\em double-level} optimization structure, updating the outer variable $x$ after fully updating $y$ over the inner and outer functions, which could lose efficiency and cause difficulty in implementations. Further, these approaches have been provided with only the {\em asymptotic} convergence guarantee without characterization of the convergence rate.

	{\em The focus of this paper is to develop a better-structured bilevel optimization algorithm, which handles the multiple inner minima challenge and comes with a finite-time convergence rate guarantee. }

	\subsection{Our Contributions}
	In this paper, we adopt a reformulation of bilevel optimization to constrained optimization \cite{dempe2020bilevel}, and 
	propose a novel primal-dual algorithm to solve the problem, which provably converges to an $\epsilon$-accurate KKT point. The specific contributions are summarized as follows.

	{\bf Algorithmic design.} We propose a simple and easy-to-implement primal-dual bilevel optimization (PDBO) algorithm, and further generalizes PDBO to its proximal version called Proximal-PDBO. 
 
Differently from existing bilevel methods designed for handling multiple inner minima in \cite{liu2020generic,li2020improved,liu2021value} that update variables $x$ and $y$ in a {\em nested} manner, both PDBO and Proximal-PDBO update $x$ and $y$ simultaneously as a single variable $z$ and hence admit a much simpler implementation. In addition, both algorithms do not involve any second-order information of the inner and outer functions $f$ and $g$, as apposed to many AID- and ITD-based approaches, and hence are computationally more efficient.

	{\bf Convergence rate analysis.} We provide the convergence rate analysis for PDBO and Proximal-PDBO, which serves as the first-known convergence rate guarantee for bilevel optimization with multiple inner-level minima. For PDBO, we first show that PDBO converges to an optimal solution of the associated constrained optimization problem under certain convexity-type conditions. Then, for nonconvex $f$ and convex $g$ on $y$, we show that the more sophisticated Proximal-PDBO algorithm achieves an $\epsilon$-KKT point of the reformulated constrained optimization problem for any arbitrary $\epsilon>0$ with a sublinear convergence rate. Here, the KKT condition serves as a necessary optimality condition for the bilevel problem. Technically, the reformulated constrained problem here is more challenging than the standard constrained optimization problem studied in \cite{boob2019stochastic,ma2019proximally} due to the nature of bilevel optimization. Specifically, our analysis needs to deal with the bias errors arising in gradient estimations for the updates of both the primal and dual variables. Further, we establish uniform upper bound on optimal dual variables, which was taken as an assumption in the standard analysis \cite{boob2019stochastic}.
	
	{\bf Empirical performance.} In two synthetic experiments with intrinsic multiple inner minima, we show that our  algorithm converges to the global minimizer, whereas AID- and ITD-based methods are stuck in local minima. We further demonstrate the effectiveness and better performance of our algorithm in hyperparameter optimization. 
	
	\subsection{Related Works}
	
	\noindent {\bf Bilevel optimization via AID and ITD.} AID and ITD are two popular approaches to reduce the computational challenging in  approximating the outer-level gradient (which is often called hypergradient in the literature). In particular, AID-based bilevel algorithms~\citep{domke2012generic,pedregosa2016hyperparameter,gould2016differentiating,liao2018reviving,grazzi2020iteration,lorraine2020optimizing,ji2021lower,mackay2019self} approximate the hypergraident efficiently via implicit differentiation combined with a linear system solver. ITD-based approaches~\citep{domke2012generic,maclaurin2015gradient,franceschi2017forward,franceschi2018bilevel,shaban2019truncated,grazzi2020iteration,mackay2019self} approximate the inner-level problem using a dynamic system. For example,~\cite{franceschi2017forward,franceschi2018bilevel} computed the hypergradient via reverse or forward mode in automatic differentiation. This paper proposes a novel contrained optimization based approach for bilevel optimization.

	\noindent {\bf Optimization theory for bilevel optimization.} Some works such as~\citep{franceschi2018bilevel,shaban2019truncated} analyzed the 
	asymptotic convergence performance of AID- and ITD-based bilevel algorithms. Other works \cite{ghadimi2018approximation,rajeswaran2019meta,grazzi2020bo,ji2020convergence,ji2021bo,ji2021lower} provided convergence rate analysis for various AID- and ITD-based approaches and their variants in applications such as meta-learning. Recent works \cite{hong2020two,ji2021bo,yang2021provably,khanduri2021near,chen2021single,guo2021randomized} developed convergence rate analysis for their proposed stochastic bilevel optimizers. This paper provides the first-known convergence rate analysis for the setting with multiple inner minima. 
	
	\noindent {\bf Bilevel optimization with multiple inner minima.} Sabach and Shtern in \cite{sabach2017first} proposed a bilevel gradient sequential averaging method (BiG-SAM) for {\em single}-variable bilevel optimization (i.e., without  variable $x$), and provided an asymptotic convergence analysis for this algorithm. For general bilevel problems, the authors in \cite{liu2020generic,li2020improved} used an idea similar to BiG-SAM, and proposed a gradient aggregation approach for the general bilevel problem in \cref{eq:prob} with an asymptotic convergence guarantee. Further, Liu et. al. \cite{liu2021value} proposed a constrained optimization method and further applied the log-barrier interior-point method for solving the constrained problem. Differently from the above studies that update the outer variable $x$ after fully updating $y$, our PDBO and Proximal-PDBO algorithms treat both $x$ and $y$ together as a single updating variable $z$. Further, we characterize the first known convergence rate guarantee for the type of bilevel problems with multiple inner minima.

We further mention that	Liu et al. in \cite{liu2021towards} proposed an initialization auxiliary algorithm for the bilevel problems with a nonconvex inner objective function.

\section{Problem Formulation}\label{sec:formulation}
	We study a bilevel optimization problem given in \cref{eq:prob}, which is restated below 
	\begin{align*}
		\min_{x \in \mathcal{X}, y \in \mathcal{S}_x} f(x, y)
		\quad\mbox{with} \quad \mathcal{S}_x = \argmin_{y \in \mathcal{Y}} g(x,y),
	\end{align*}
	where the outer- and inner-level objective functions $f(x, y)$ and $g(x, y)$ are continuously differentiable, and the supports $\mathcal{X}$ and $\mathcal{Y}$ are convex and closed subsets of $\mathbb{R}^p$ and $\mathbb{R}^d$, respectively. For a fixed $x\in\mathcal{X}$, $\mathcal{S}_x$ is the set of all $y\in\mathcal{Y}$ that yields the minimal value of $g(x, \cdot)$. In this paper, we consider the function $g$ that is a convex function on $y$ for any fixed $x$ (as specified in \Cref{ass:smoothness}).  
	The convexity of $g(x, y)$ on $y$ still allows  the inner function $g(x, \cdot)$ to have multiple global minimal points, and the challenge for bilevel algorithm design due to multiple inner minima still remains. 
	Further, the set $\mathcal{S}_x$ of minimizers is convex due to convexity of $g(x, y)$ w.r.t.\ $y$. We note that $g(x,y)$ and the outer function $f(x,y)$ can be nonconvex w.r.t $(x, y)$ in general or satisfy certain convexity-type conditions which we will specify for individual cases. We further take the standard gradient Lipschitz assumption on the inner and outer objective functions. The formal statements of our assumptions are presented below.
	\begin{assumption}\label{ass:smoothness}
		The objective functions $f(x,y)$ and  $g(x,y)$ are gradient Lipschitz functions. There exists $\rho_f, \rho_g\ge 0$, such that, for any $z=(x, y)\in\mathcal{X}\times\mathcal{Y}$ and $z^\prime=(x^\prime, y^\prime)\in\mathcal{X}\times\mathcal{Y}$, the following inequalities hold
		\begin{align*}
				&\|\nabla f(z) - \nabla f(z^\prime)\|_2 \le \rho_f\|z - z^\prime\|_2,
				&\|\nabla g(z) - \nabla g(z^\prime)\|_2 \le \rho_g\|z -z^\prime\|_2.
		\end{align*}
		Moreover, for any fixed $x\in\mathcal{X}$, we assume $g(x,y)$ is a convex function on $y$, and the following inequality holds for all $x, x^\prime\in\mathcal{X}$ and $y, y^\prime\in\mathcal{Y}$
		\begin{equation*}
		    g(x^\prime, y^\prime)\ge g(x, y) + \langle\nabla_x g(x, y), x^\prime-x\rangle + \langle \nabla_y g(x, y), y^\prime -y\rangle - \tfrac{\rho_g}{2}\|x - x^\prime\|_2^2.
		\end{equation*}
 
	\end{assumption}

	To solve the bilevel problem in \cref{eq:prob}, one challenge is that it is not easy to explicitly characterize the set $\cS_x$ of the minimal points of $g(x,y)$. This motivates the idea to describe such a set implicitly via a constraint. A common practice is to utilize the so-called lower-level value function (LLVF) to reformulate the problem to an equivalent single-level optimization~\cite{dempe2020bilevel}. Specifically, let $g^*(x) \coloneqq \min_{y\in\mathcal{Y}} g(x, y)$. Clearly, the set $\mathcal{S}_x$ can be described as $\mathcal{S}_x=\{y\in \mathcal{Y}: g(x,y)\leq g^*(x)\}$. In this way, the bilevel problem in \cref{eq:prob} can be equivalently reformulated to the following constrained optimization problem:
	\begin{align*}
		\min_{x\in\mathcal{X},y\in\mathcal{Y}} f(x,y)\quad \text{s.t.} \quad  g(x,y)\leq g^*(x).
	\end{align*}
	Since $g(x,y)$ is convex with respect to $y$, $g^*(x)$ in the constraint can be obtained efficiently via various convex minimization algorithms such as gradient descent. 
	
	To further simplify the notation, we let $z = (x, y)\in\mathbb{R}^{p+d}$, $\mathcal{Z} = \mathcal{X}\times\mathcal{Y}$, $f(z):=f(x,y)$, $g(z):=g(x,y)$, and $g^*(z):= g^*(x) $. 
	As stated earlier, both $\mathcal{X}$ and $\mathcal{Y}$ are bounded and closed set. The boundedness of $\mathcal{Z}$ is then immediately established, and we denote $D_\mathcal{Z} =\sup_{z, z^\prime\in\mathcal{Z}}\|z-z^\prime\|_2$.  The equivalent single-level constrained optimization could be expressed as follows.
	\begin{align}\label{eq:reformulation}
		\min_{z\in\mathcal{Z}} f(z) \quad \mbox{s.t. }\quad h(z) \coloneqq g(z) -  g^*(z)\le 0.
	\end{align}
	Thus, solving the bilevel problem in \cref{eq:prob} is converted to solving an equivalent single-level optimization in \cref{eq:reformulation}.  
	To enable the algorithm design for the constrained optimization problem in \cref{eq:reformulation}, we further make two standard changes to the constraint function. 
	(i) Since the constraint is nonsmooth, i.e., $g^*(z):=g^*(x)$ is nonsmooth in general, the design of gradient-based algorithm is not direct.  
	We thus relax the constraint by replacing $g^*(z)$ with a smooth term \[\tilde{g}^*(z):=\tilde{g}^*(x) = \min\nolimits_{y \in \mathcal{Y}} \big\{\tilde{g}(x,y) \coloneqq g(x, y) + \tfrac{\alpha}{2}\|y\|^2\big\},\] 
	where $\alpha>0$ is a small prescribed constant.  
	It can be shown that for a given $x$, $\tilde{g}(x,y)$ has a unique minimal point, and the function $\tilde{g}^*(x)$ becomes differentiable with respect to $x$. Hence, the constraint becomes $g(z)- \tilde{g}^*(z)\leq 0$, which is differentiable. 
	(ii) The constraint is not sufficiently strictly feasible, i.e., the constraint cannot be satisfied with a certain margin on every $x\in\mathcal{X}$, due to which it is difficult to design an algorithm with convergence guarantee. We hence further relax the constraint by introducing a positive small constant $\delta$ so that the constraint becomes $g(z) -  \tilde{g}^*(z) - \delta \le 0$ that admits strict feasible points such that the constraint is less than $-\delta$.  
Given the above two relaxations, our algorithm design will be based on the following best-structured constrained optimization problem
	\begin{align}\label{eq:reform3}
		\min_{z\in\mathcal{Z}} f(z)  \quad
		\mbox{s.t.}\quad \tilde{h}(z) \coloneqq g(z) -  \tilde{g}^*(z) - \delta \le 0.
	\end{align}

\section{Primal-Dual Bilevel Optimizer (PDBO)}\label{sec:pd}

In this section, we first propose a simple PDBO algorithm, and then show that PDBO converges under certain convexity-type conditions. We handle more general $f$ and $g$ in \Cref{sec:proximalpdbo}. 
	
\subsection{PDBO Algorithm}	
	To solve the constrained optimization problem in \cref{eq:reform3}, we employ the primal-dual approach. The idea is to consider 
	the following dual problem 
	\begin{align}\label{eq:minimax1}
		\max_{\lambda\ge 0}\min_{z\in\mathcal{Z}}\ \mathcal{L}(z, \lambda) = f(z) + \lambda \tilde{h}(z),
	\end{align}
	where $\mathcal{L}(z, \lambda)$ is called the Lagrangian function and $\lambda$ is the dual variable. A simple approach to solving the minimax dual problem in \cref{eq:minimax1} is via the gradient descent and ascent method, which yields our algorithm of primal-dual bilevel optimizer (PDBO) (see \Cref{alg:saddle}).
	
		\begin{algorithm}[h]
		\caption{Primal-Dual Bilevel Optimizer (PDBO)}  
		\small
		\label{alg:saddle}
		\begin{algorithmic}[1]
			\STATE {\bfseries Input:}  Stepsizes $\eta_t$ and $\tau_t$, $\theta_t$, output weights $\gamma_t$, initialization $z_0, {\lambda}_0$, and number $T$ of iterations  
			\FOR{$t=0,1,...,T-1$} 
			\STATE{Conduct projected gradient descent in \cref{eq:projectedgradientdescent} for $N$ times with any given $\hat{y}_0$ as initialization}
			\STATE{Update $\lambda_{t+1}$ according to \cref{eq:lambdaupdate1}} 
			\STATE{Update $z_{t+1}$ according to \cref{eq:zupdate1}}\label{step:z} 
			\ENDFOR
			\STATE {\bfseries Output:} $\bar{z} = \tfrac{1}{\Gamma_T}\sum_{t=0}^{T-1}\gamma_tz_{t+1}$, with $\Gamma_T = \sum_{t=0}^{T-1}\gamma_t$
		\end{algorithmic}
	\end{algorithm}

	More specifically, the update of the dual variable $\lambda$ is via the gradient of Lagrangian w.r.t.\ $\lambda$ given by ${\nabla}_{\lambda} \mathcal{L} (z, \lambda)=\tilde h(z)$, and the update of the primal variable $z$ is via the gradient of Lagrangian w.r.t.\ $z$ given by ${\nabla}_{z} \mathcal{L} (z, \lambda)=\nabla f(z) + \lambda \nabla\tilde h(z)$.  
	Here, the differentiability of $\tilde{h}(z)$ benefits from the constraint smoothing. In particular, suppose \Cref{ass:smoothness} holds. Then, it can be easily shown that $\tilde{h}(z)$ is differentiable and  
	$ \nabla_x \tilde{h}(x,y) = \nabla_x g(x,y) - \nabla_x g(x,\tilde{y}^*(x))$, 
	where $\tilde{y}^*(x) = \argmin_{y \in \mathcal{Y}} \tilde{g}(x,y)$ is the unique minimal point of $\tilde{g}(x,y)$. Together with the fact that $\nabla_y \tilde{h}(z) = \nabla_y g(x, y)$, we have $\nabla \tilde{h}(z) = \left(\nabla_x g(x,y) - \nabla_x g(x,\tilde{y}^*(x)); \nabla_y g(x,y)\right)$. Since $\tilde{y}^*(x)$ is the minimal point of the inner problem: $\min_{y\in\mathcal{Y}} g(x_t, y) + \tfrac{\alpha}{2}\|y\|_2^2$, we conduct $N$ steps of projected gradient descent
	\begin{equation}\label{eq:projectedgradientdescent}
		\hat{y}_{n+1} = \Pi_{\mathcal{Y}}\big(\hat{y}_n -  \tfrac{2}{\rho_g+ 2\alpha}\left(\nabla_y g(x_t, \hat{y}_n) + \alpha \hat{y}_n\right)\big),  
	\end{equation} 
	and take $\hat{y}_N$ as an estimate of $\tilde{y}^*(x_t)$. Since the inner function $\tilde{g}(x,y)$ is $\alpha$-strongly convex w.r.t.\ $y$, updates in \cref{eq:projectedgradientdescent} converge exponentially fast to $\tilde{y}^*(x_t)$ w.r.t.\ $N$. Hence, with only a few steps, we can obtain a good estimate. With the output $\hat{y}_N$ of \cref{eq:projectedgradientdescent} as the estimate of $\tilde{y}^*(x_t)$, we conduct the accelerated projected gradient ascent and projected gradient descent as follows:
	\begin{align}
		\lambda_{t+1} &= \Pi_\Lambda\big(\lambda_t + \tfrac{1}{\tau_t}\big((1+\theta_t)\hat{h}(z_t) - \theta_t\hat{h} (z_{t-1})\big)\big),\label{eq:lambdaupdate1}\\ 
		z_{t+1} &=\Pi_{\mathcal{Z}} \big(z_t - \tfrac{1}{\eta_t}\big(\nabla f(z_t) + \lambda_{t+1}\hat{\nabla}\tilde{h}(z_t)\big) \big),\label{eq:zupdate1}
	\end{align}
	where $\tfrac{1}{\tau_t}$, $\tfrac{1}{\eta_t}$ are the stepsizes, $\theta_t$ is the acceleration weight, $\hat{h}(z_t)=g(x_t, y_t) - \tilde{g}(x_t, \hat{y}_N)$, $\hat{\nabla} \tilde{h}(z_t)= \nabla g(z_t) - \left(\nabla_x g(x_t, \hat{y}_N);\mathbf{0}_d\right)$, and $\Lambda = [0, B]$, with $B>0$ being a prescribed constant.

\subsection{Convergence Rate of PDBO}

As formulated in \Cref{sec:formulation}, the problem in \cref{eq:reform3} in general can be a nonconvex objective and nonconvex constrained optimization under \Cref{ass:smoothness}. For such a problem, the gradient descent with respect to $z$ can guarantee only the convergence $\|\nabla_z \mathcal{L}(z, \lambda)\|_2\to 0$. Here, the updates of $\lambda$ change only the weight that the gradient of the constraint contributes to the gradient of the Lagrangian, which does not necessarily imply the convergence of the function value of the constraint. Thus, we anticipate PDBO to converge under further geometric requirements as stated below. 
\begin{assumption}\label{ass:strongassp}
     The objective function $f(z)$ is a $\mu$-strongly convex function with respect to $z$, and the constrained function $\tilde{h}(z)$ is a convex function on $z$.  
\end{assumption} 
Under \Cref{ass:strongassp}, the global optimal point exists and is unique. Let such a point be $z^*$. We provides the convergence result with respect to such a point below.   
\begin{theorem}\label{thm:pdboconvergence}
    Suppose \Cref{ass:smoothness,ass:strongassp} hold. Consider \Cref{alg:saddle}.  Let $B>0$ be some large enough constant, $\gamma_t = \mathcal{O}(t)$, $\eta_t = \mathcal{O}(t)$, $\tau_t=\mathcal{O}(\tfrac{1}{t})$ and $\theta_t = \gamma_{t+1}/\gamma_t$, where the exact expressions can be found in the appendix. Then, the output $\bar{z}$ of PDBO converges to $z^*$, which satisfies
		\begin{equation*}
			\max\{f(\bar{z}) - f(z^*), [\tilde{h}(\bar{z})]_+, \|\bar{z} - {z}^*\|_2^2\}	\le \mathcal{O}\left(\tfrac{1}{T^2}\right) + \mathcal{O}\left({e^{-N}}\right), \mbox{ with } [x]_+=\max\{x, 0\}.
		\end{equation*}
\end{theorem}
\Cref{thm:pdboconvergence} indicates that all of the optimality gap $f(\bar{z}) - f(z^*)$, the constraint violation $[\tilde{h}(\bar{z})]_+$, and the squared distance $\|\bar{z} - {z}^*\|_2^2$ between the output and the optimal point converge sublinearly as the number of iterations enlarges. In particular, the first term of the bound captures the accumulated gap of the outer function values among iterations, and the second term is due to the biased estimation of  $\tilde{y}^*(x_t)$ in each iteration.  

\begin{corollary}
By setting $T= \mathcal{O}(\tfrac{1}{\sqrt{\epsilon}})$ and $N = \mathcal{O}(\log(\tfrac{1}{\epsilon}))$, \Cref{thm:pdboconvergence} ensures that $\bar{z}$ is an $\epsilon$-optimal point of the constrained problem in \cref{eq:reform3}, i.e. the optimality gap $f(\bar z) - f(z^*)$, constraint violation $[\tilde{h}(\bar z)]_+$, and squared distance $\|\bar{z} - z^*\|_2^2$ are all upper-bounded by $\epsilon$. Moreover, the total complexity of gradient accesses is given by $TN= \tilde{O}(\tfrac{1}{\sqrt{\epsilon}})$.  
\end{corollary}

We remark that our proof here is more challenging than the generic constrained optimization \cite{boob2019stochastic,ma2019proximally} due to the nature of the bilevel optimization. Specifically, the constraint function here includes the minimal value $y^*(x_t)\coloneqq\argmin_{y\in\mathcal{Y}} \tilde{g}(x_t, y)$ of the inner function, where both its value and the minimal point will be estimated during the execution of algorithm, which will cause the gradients of both primal and dual variables to have bias errors. Our analysis will need to deal with such bias errors and characterize their impact on the convergence.

We further remark that although PDBO has guaranteed convergence under convexity-type conditions, it can still converge fast under more general problems when $f$ and $g$ are nonconvex as we demonstrate in our experiments in \Cref{sec:experiment} and in appendix. However, formal theoretical treatment of nonconvex problems will require more sophisticated design as we present in the next section.

\section{Proximal-PDBO Algorithm} \label{sec:proximalpdbo}

\subsection{Algorithm Design}
In the previous section, we introduce PDBO and provide its convergence rate under convexity-type conditions. In order to solve the constrained optimization problem \cref{eq:reform3} in the general setting, we will adopt the proximal method \cite{boob2019stochastic,ma2019proximally}. 
	The general idea is to iteratively solve a series of sub-problems, constructed by regularizing the objective and constrained functions into strongly convex functions. In this way, the algorithm is expected to converge to a stochastic $\epsilon$-KKT point (see \Cref{def:KKT} in \Cref{sec:analysisofproxpdbo}) of the primal problem in \cref{eq:reform3}.
	
	By applying the proximal method, we obtain the Proximal-PDBO algorithm (see \Cref{alg:proximalmethod}) for solving the bilevel optimization problems formulated in \cref{eq:reform3}. At each iteration, the algorithm first constructs two proximal functions corresponding to the objective $f(z)$ and constraint $\tilde{h}(z)$ via regularizers. Since $f(z)$ is $\rho_f$-gradient Lipschitz as assumed in \Cref{ass:smoothness}, the constructed function $f_k(z)$ is strongly convex with $\mu=\rho_f$. For the new constraint $\tilde{h}_k(z)$, we next show that it is a convex function with large enough regularization coefficient $\rho$.
	\begin{lemma}\label{lemma:lipschizh}
		Suppose that \Cref{ass:smoothness} holds. Let $\rho = \tfrac{2\alpha\rho_g + \rho_g^2}{2\alpha}$, then $\tilde{h}_k(z)$ is a convex function.
	\end{lemma}
	The above lemma and the $\rho_f$-strong convexity of $f_k(z)$ ensure that \Cref{ass:strongassp} holds with $f_k(z)$ and $h_k(z)$. Then lines 5-10 adopt the PDBO as a subroutine to solve the subproblem $(\mathrm{P}_k)$, which is shown to be effective in \Cref{thm:pdboconvergence}. Similarly to what we have done in \Cref{sec:pd}, we estimate $\tilde{h}_k(z_t)$ and $\nabla \tilde{h}_k(z_t)$ using $\hat{y}_N$ as $\hat{h}_k(z_t)= g(x_t, y_t) - \tilde{g}(x_t, \hat{y}_N) -\delta + \rho_h\|x_t - \tilde{x}_k\|_2^2$, and $\hat{\nabla}  \tilde{h}_k(z_t) = \nabla g(z_t) - (\nabla_x \tilde{g}(x_t, \hat{y}_N); \mathbf{0}_d) + 2\rho(x_t - \tilde{x}_k; \mathbf{0}_d)$.
	The gradient of the Lagrangian is immediately obtained through $\hat{\nabla}_{\lambda} \mathcal{L}_k (z_t, \lambda_{t})= \hat{h}_k(z_t)$ and $\hat{\nabla}_z \mathcal{L}_k (z_t, \lambda_{t+1})= \nabla f_k(z_t) + \lambda_{t+1}\hat{\nabla}\tilde{h}_k(z_t)$. Finally, the main step of updating dual and primal variables in \cref{eq:zupdate1,eq:lambdaupdate1} are adjusted here as follows: 
	\begin{align}
		\lambda_{t+1} &= \Pi_{\Lambda}\big(\lambda_t + \tfrac{1}{\tau_t}\big((1+\theta_t)\hat{h}_k(z_t) - \theta_t\hat{h}_k (z_{t-1}))\big)\big),\label{eq:accelerateupdatelambda}\\
		z_{t+1} &= \Pi_{\mathcal{Z}}\big(z_t -\tfrac{1}{\eta_t}\hat{\nabla}_z \mathcal{L}_k (z_t, \lambda_{t+1})\big).\label{eq:zupdatemain}
	\end{align} 
	
	\begin{algorithm}[tb]
		\caption{Proximal-PDBO Algorithm}
		\small
		\label{alg:proximalmethod}
		\begin{algorithmic}[1]
			\STATE {\bfseries Input:} Stepsizes $\eta_t$ and $\tau_t$, $\theta_t$, output weights $\gamma_t$  and iteration numbers $K$, and $T$
			\vspace{0.05cm}
			\STATE Set $\tilde{z}_0$ be any point inside $\mathcal{Z}$
			\FOR{$k=1,...,K$}
			\STATE Set the sub-problem
			\begin{equation}
					\min_{z\in\mathcal{Z}} \ f_k(z) \coloneqq f(z) + \rho_f \|z - \tilde{z}_{k-1}\|_2^2, \ \mbox{s.t.}\  \tilde{h}_k(z) \coloneqq \tilde{h}(z) + \rho \|x- \tilde{x}_{k-1}\|_2^2 \le 0. \tag{$\mbox{P}_k$}
			\end{equation} 
			\STATE Internalize $z_0 = z_{-1} = \tilde{z}_{k-1}$ and $\lambda_0=\lambda_{-1}=0$
			\FOR{$t=0,1,...,T-1$}
			\vspace{0.05cm}
			\STATE Conduct updates of $\hat{y}_t$ in \cref{eq:projectedgradientdescent} for $N$ times with any initial point $\hat{y}_0\in\mathcal{Y}$ to estimate $\tilde{y}^*(x_t)$ 
			\STATE{Update $\lambda_{t+1}$ according \cref{eq:accelerateupdatelambda}} 
			\vspace{0.05cm}
			\STATE{Update $z_{t+1}$ according to \cref{eq:zupdatemain}}  
			\vspace{0.05cm}
			\ENDFOR
			\STATE Set $\tilde{z}_k = \tfrac{1}{\Gamma_T}\sum_{t=0}^{T-1} \gamma_t z_{t+1}$, with $\Gamma_T =\sum_{t=0}^{T-1}\gamma_t$.
			\ENDFOR
			\STATE Randomly pick $\hat{k}$ from $\{1, \ldots, K\}$
			\STATE {\bfseries Output:}  $\tilde{z}_{\hat{k}}$
		\end{algorithmic}
	\end{algorithm}
	
	\subsection{Convergence Rate of Proximal-PDBO}\label{sec:analysisofproxpdbo}
	
	We first introduce the following first-order necessary condition of optimality for the nonconvex optimization problem with nonconvex constraint in \cref{eq:reform3}  \cite{boob2019stochastic,ma2019proximally}. 
	\begin{definition}[(Stochastic) $\epsilon$-KKT point]\label{def:KKT}
		Consider the constrained optimization problem in \cref{eq:reform3}. A point $\hat{z}\in\mathcal{Z}$ is an $\epsilon$-KKT point iff., there exist $z\in\mathcal{Z}$ and $\lambda\ge 0$ such that $\tilde h(z)\le 0$, $\|z-\hat{z}\|_2^2\le \epsilon$, $|\lambda \tilde h(z)|\le \epsilon$, and $\mathrm{dist}\left(\nabla f(z) + \nabla \tilde h(z), -\mathcal{N}(z; \mathcal{Z})\right)\le \epsilon$, where $\mathcal{N}(z; \mathcal{Z})$ is the normal cone to $\mathcal{Z}$ at $z$, and the distance between a vector $v$ and a set $\mathcal{V}$ is $\mathrm{dist}(v, \mathcal{V})\coloneqq\inf\{\|v - v^\prime\|_2: v^\prime\in\mathcal{V}\}$. For random $\hat{z}\in\mathcal{Z}$, it is a stochastic $\epsilon$-KKT point if there exist $z\in\mathcal{Z}$ and $\lambda\ge 0$ such that the same requirements of $\epsilon$-KKT hold in expectation.
	\end{definition}
We will take the $\epsilon$-KKT condition as our convergence metric. It has been shown that the above KKT condition serves as the first-order necessary condition for the optimality guarantee for nonconvex  optimization with nonconvex constraints under the MFCQ condition (see more details in \cite{mangasarian1967fritz}).

Next, we establish the convergence guarantee for Proximal-PDBO, which does not follow directly from that for standard constrained nonconvex optimization \cite{boob2019stochastic,ma2019proximally} due to the special challenges arising in bilevel problem formulations. Our main development lies in showing that the optimal dual variables for all subproblems visited during the algorithm iterations are uniformly bounded. Then the convergence of the Proximal-PDBO follows from the convergence of each subproblem (which we establish in \Cref{thm:pdboconvergence}) and the uniform bound of optimal dual variables. The details of the proof could be found in the appendix.  

\begin{theorem}\label{thm:maintheorem}
	Suppose \Cref{ass:smoothness} holds. Consider \Cref{alg:proximalmethod}. Let the hyperparameters $B>0$ be a large enough constant, $\gamma_t=\mathcal{O}(t)$, $\eta_t=\mathcal{O}(t)$, $\tau_t=\mathcal{O}(\frac{1}{t})$ and $\theta_t=\gamma_{t+1}/\gamma_t$. Then, the output $\tilde{z}_{\hat{k}}$ of \Cref{alg:proximalmethod} with a randomly chosen index $\hat k$ is a stochastic $\epsilon$-KKT point of \cref{eq:reformulation}, where $\epsilon$ is given by
	$\epsilon=\mathcal{O}\left(\tfrac{1}{K}\right)+\mathcal{O}\left(\tfrac{1}{T^2}\right)+ \mathcal{O}\left(e^{-N}\right)$.
\end{theorem}
\Cref{thm:maintheorem} characterizes the convergence of Proximal-PDBO. In particular, there are three sources of the convergence error: (a) the inaccurate initialization of proximal center $\tilde{z}_0$ captured by $\mathcal{O}(\tfrac{1}{K})$, (b) the distance between $z_0$ and the optimal point of each subproblem $(\mathrm{P}_k)$ upper-bounded by $\mathcal{O}(\frac{1}{T^2})$, and (c) the inaccurate estimation of $\tilde y^*(x_t)$ in the updates captured by $\mathcal{O}(e^{-N})$. 

\begin{corollary}
\Cref{thm:maintheorem} indicates that for any prescribed accuracy level $\epsilon>0$, by setting $K = \mathcal{O}(\tfrac{1}{\epsilon})$, $T=\mathcal{O}(\tfrac{1}{\sqrt{\epsilon}})$ and $N = \mathcal{O}(\log(\tfrac{1}{\epsilon}))$, we obtain an $\epsilon$-KKT point in expectation. The total computation of gradients is given by $KTN=\tilde{\mathcal{O}}(\tfrac{1}{\epsilon^{3/2}})$. 
\end{corollary} 

We further note that \Cref{thm:pdboconvergence,thm:maintheorem} are the first known finite-time convergence rate characterization for bilevel optimization problems with multiple inner minimal points.  

\section{Experiments}\label{sec:experiment}
In this section, we first consider the numerical verification for our algorithm over two synthetic problems, where one of them is moved to the appendix due to the page limit, and then apply it to hyperparameter optimization.

\subsection{Numerical Verification} 
Consider the following bilevel optimization problem: 
\begin{equation} \label{eq:toy} 
	\min\nolimits_{x \in \mathbb{R}} \mbox{ }  f(x,y) \coloneqq \tfrac{1}{2} \| (1, x)^\top - y \|^2 \quad \mbox{s.t.} \quad y \in \argmin\nolimits_{y \in \mathbb{R}^2} \mbox{ } g(x,y) \coloneqq \tfrac{1}{2}y_1^2 - x y_1, 
\end{equation}
where $y$ is a vector in $\mathbb{R}^2$ and $x$ is a scalar.  
It is not hard to analytically derive that the optimal solution of the problem in \cref{eq:toy} is $(x^*, y^*) = (1, (1,1))$, which corresponds to the optimal objective values of $f^* = 0$ and $g^* = -\frac{1}{2}$. 
For a given value of $x$, the lower-level problem admits a unique minimal value $g^*(x) =  -\frac{1}{2} x $, which is attained at all points $y = (x, a)^\top$ with $a \in \mathbb{R}$. Hence, the problem in \cref{eq:toy} violates the requirement of the existence of a \textbf{single} minimizer for the inner-problem, which is a strict requirement for most existing bilevel optimization methods, but still fall into our theoretical framework that allows multiple inner minimizers. In fact, it can be analytically shown that standard AID and ITD approaches cannot solve the problem in \cref{eq:toy} \citep{liu2020generic}, which makes it both interesting and challenging. 

We compare our algorithms \textbf{(Proximal-)PDBO} with the following representative methods for bilevel optimization: 
\begin{list}{$\bullet$}{\topsep=0.1ex \leftmargin=0.15in \rightmargin=0.1in \itemsep =0.01in}
	\item \textbf{BigSAM + ITD} \citep{liu2020generic, li2020improved}: uses sequential averaging to solve the inner problem and applies reverse mode automatic differentiation to compute hypergradient. This method is also designed to solve bilevel problems with multiple inner minima. 
	\item \textbf{AID-FP} \cite{grazzi2020iteration}: an approximate implicit differentiation approach with Hessian inversion using fixed point method. This method is guaranteed to converge when the lower-level problem admits a unique minimizer. 
	\item \textbf{ITD-R}\citep{franceschi2017forward}: the standard iterative differentiation method for bilevel optimization, which differentiates through the unrolled inner gradient descent steps. We use its reverse mode implementation. This method is also guaranteed to converge when the lower level problem admits a unique minimizer. 
\end{list}  

\begin{figure}
\centering
\end{figure}
	\begin{figure*}[ht]
		\centering
		\begin{tabular}{ccc}
			\small{(a) outer objective v.s. iterations}  & \small{(b) inner opt. gap v.s. iterations}  & \small{(c) gradient norm v.s. iterations} \\
			\includegraphics[width=4.2cm,height=2.8cm]{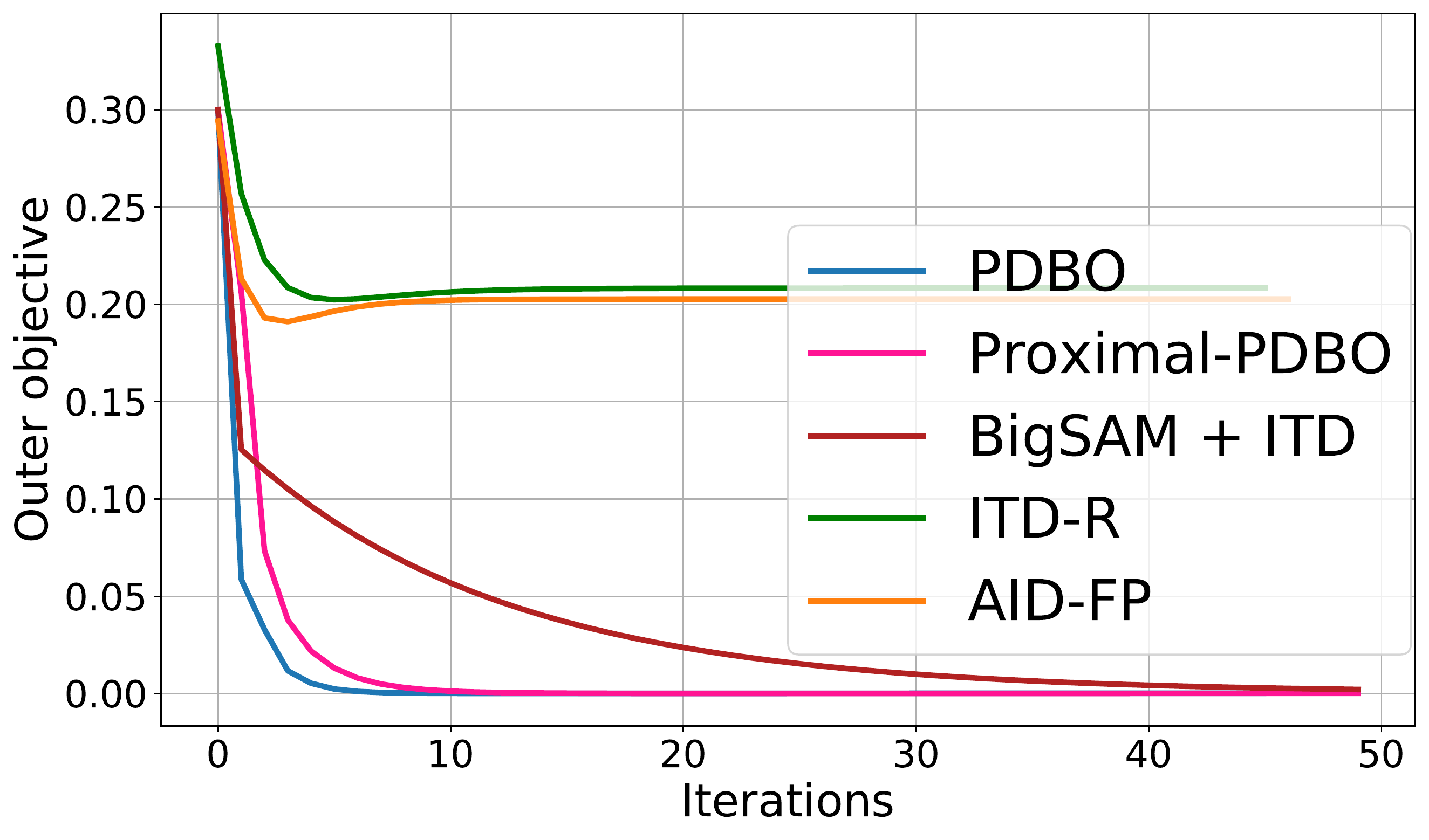}
			&\includegraphics[width=4.2cm,height=2.8cm]{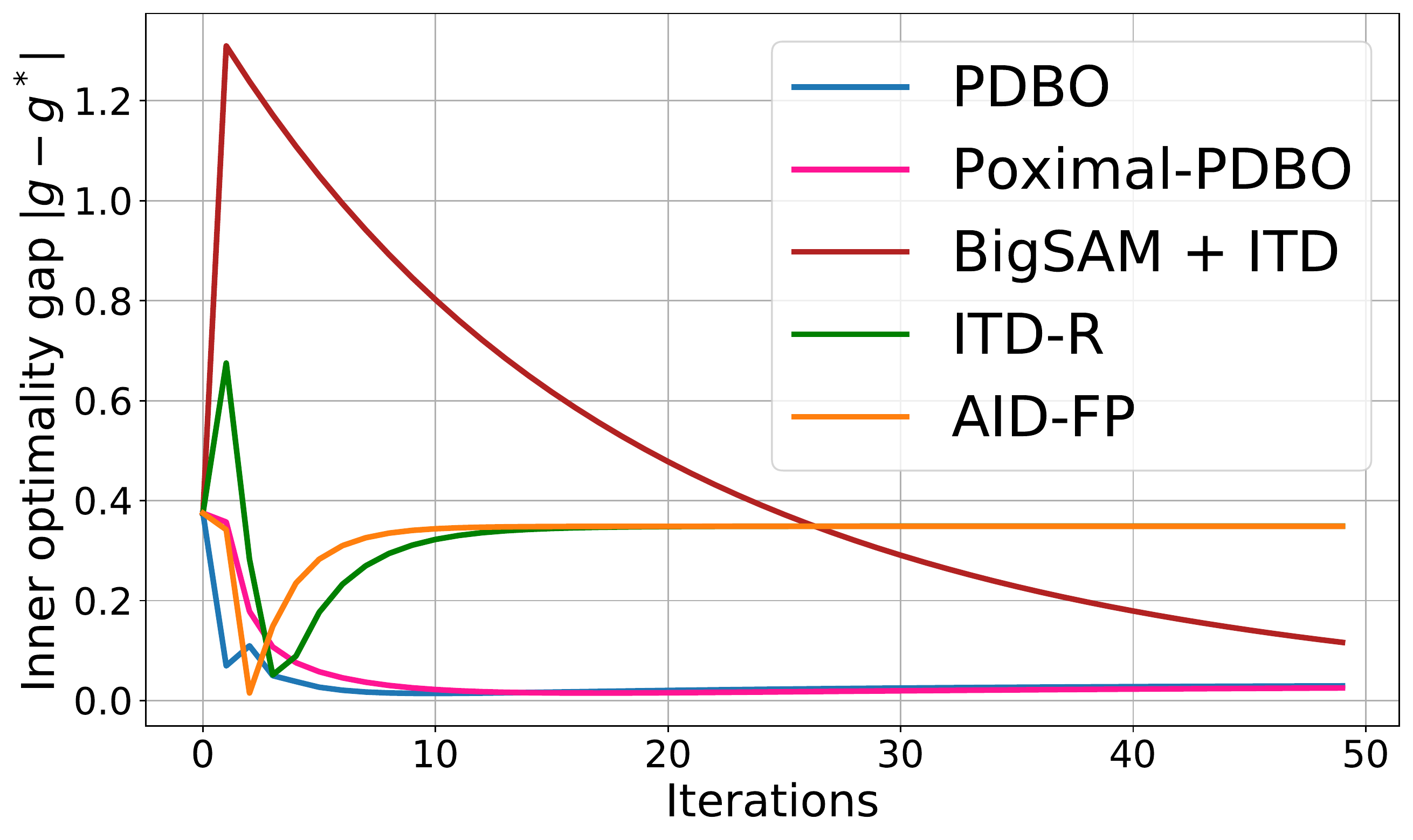}
			&\includegraphics[width=4.2cm,height=2.8cm]{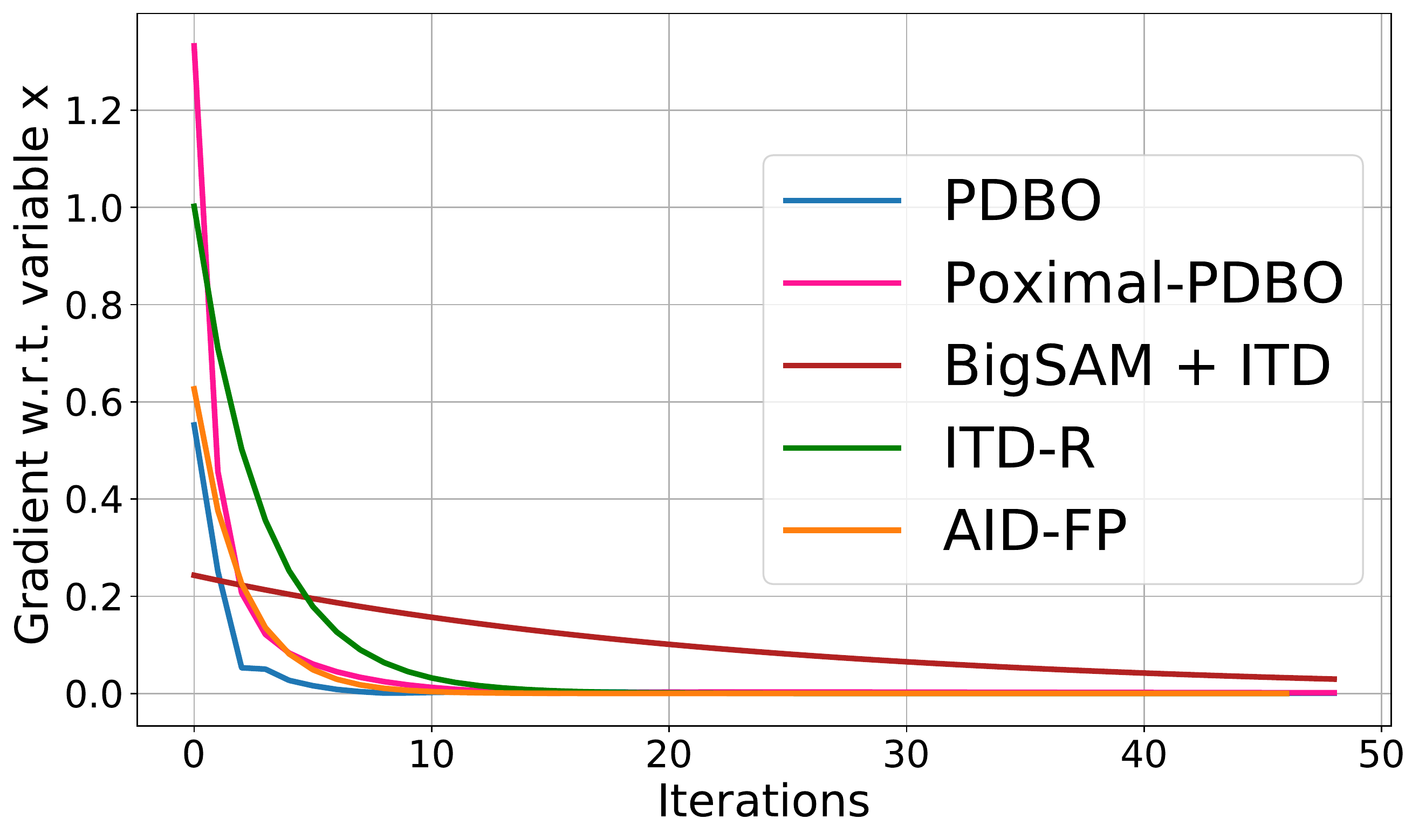}\\
			\small{(d) $\|y - y^*\|$ v.s. iterations} & \small{(e) $\|x - x^*\|$ v.s. iterations} & \small{(f) outer objective v.s. iterations}  \\
			\includegraphics[width=4.2cm,height=2.8cm]{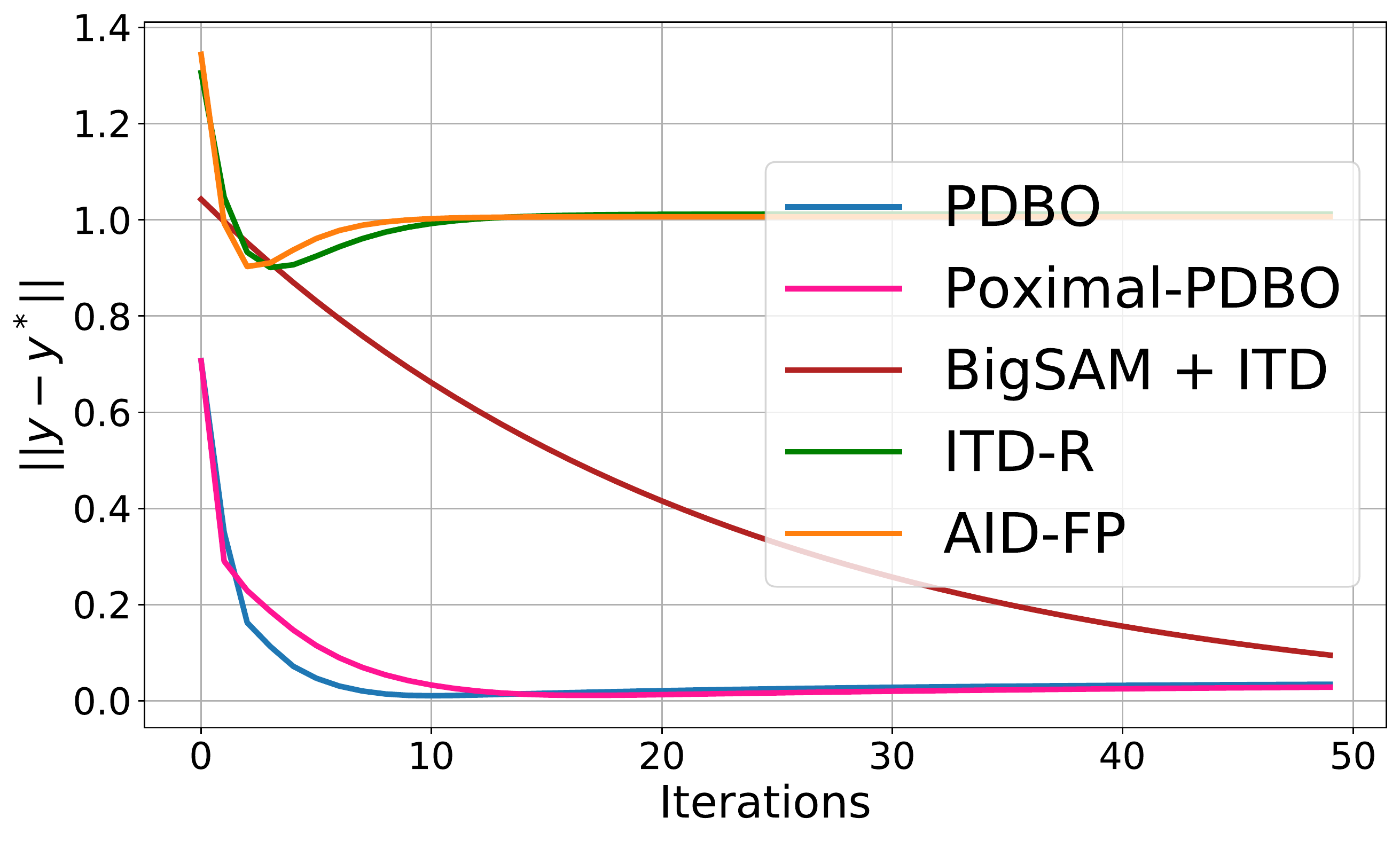}
			&\includegraphics[width=4.2cm,height=2.8cm]{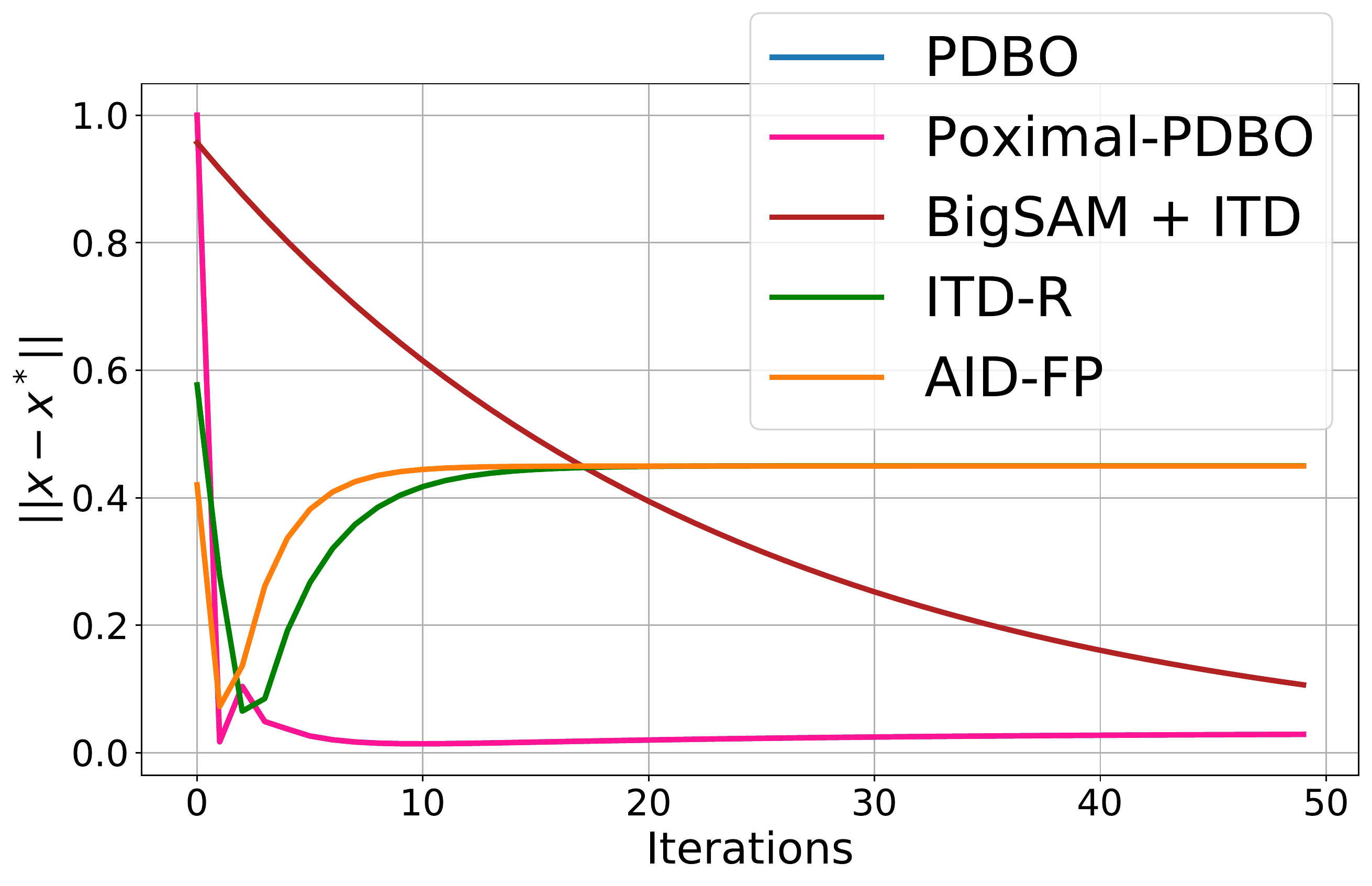}
			&\includegraphics[width=4.2cm,height=2.8cm]{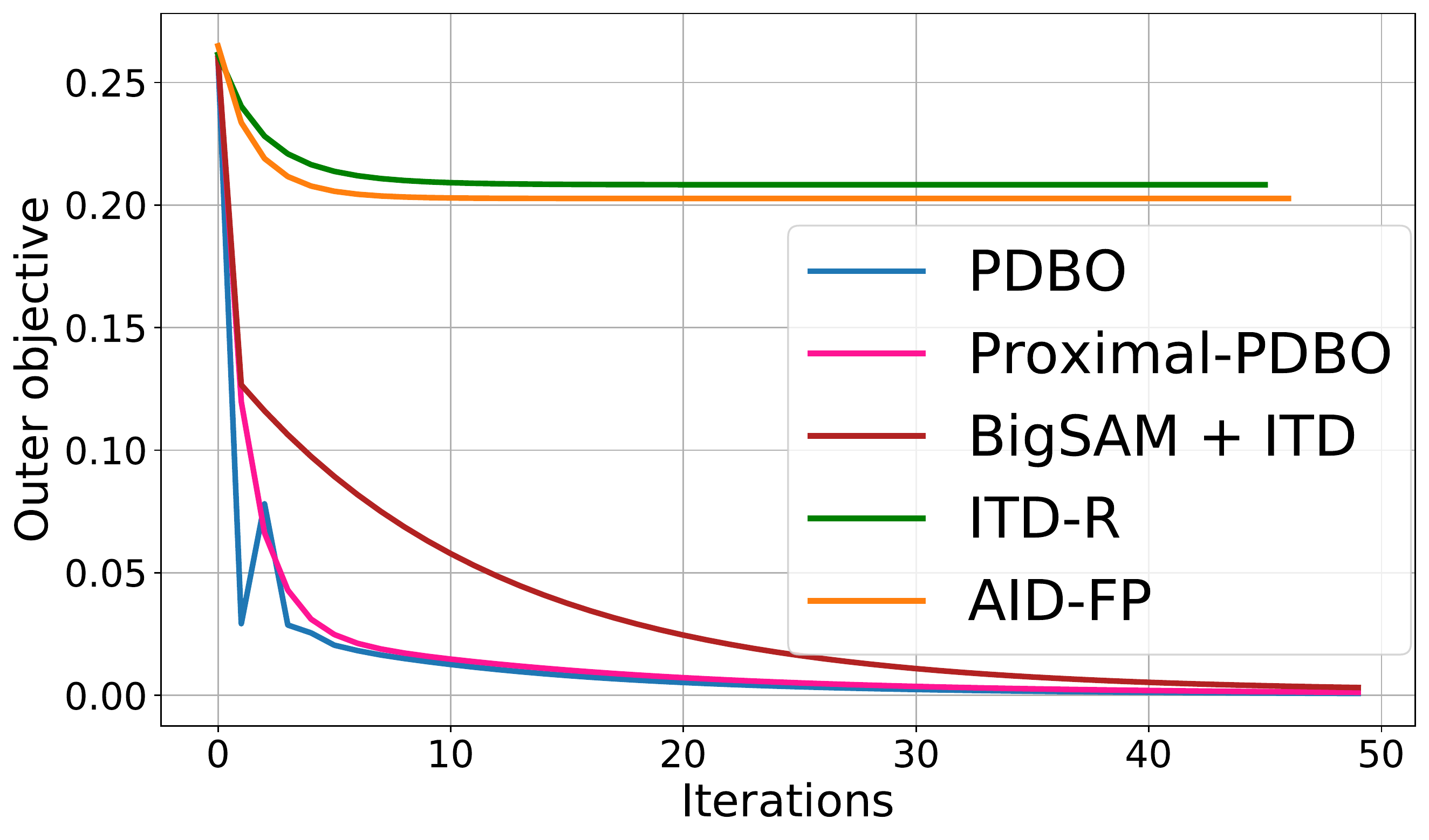}\\ 
			\small{(g) inner opt. gap v.s. iterations} & \small{(h) $\|y - y^*\|$ v.s. iterations} & \small{(i) $\|x - x^*\|$ v.s. iterations} \\
			\includegraphics[width=4.2cm,height=2.8cm]{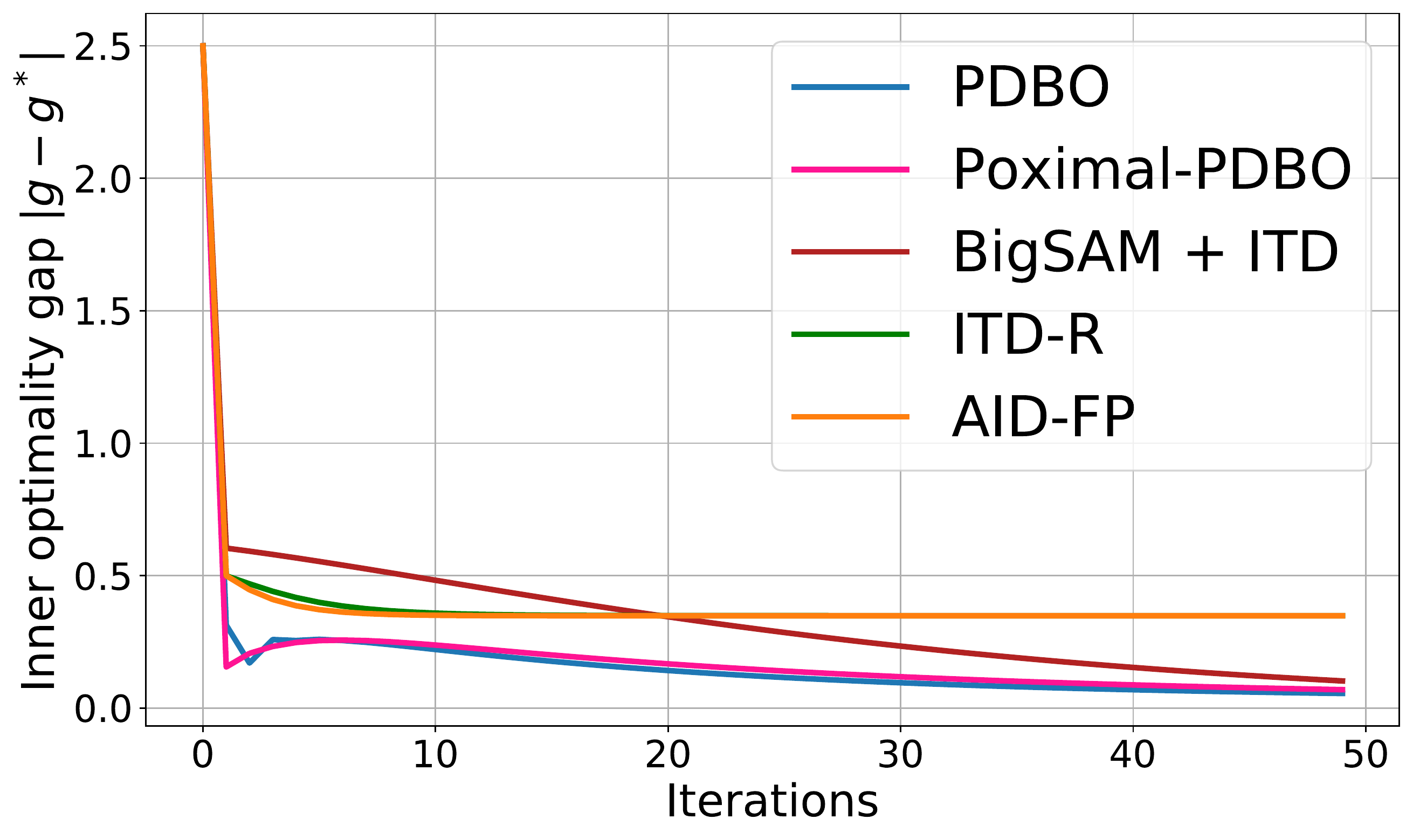}
			&\includegraphics[width=4.2cm,height=2.8cm]{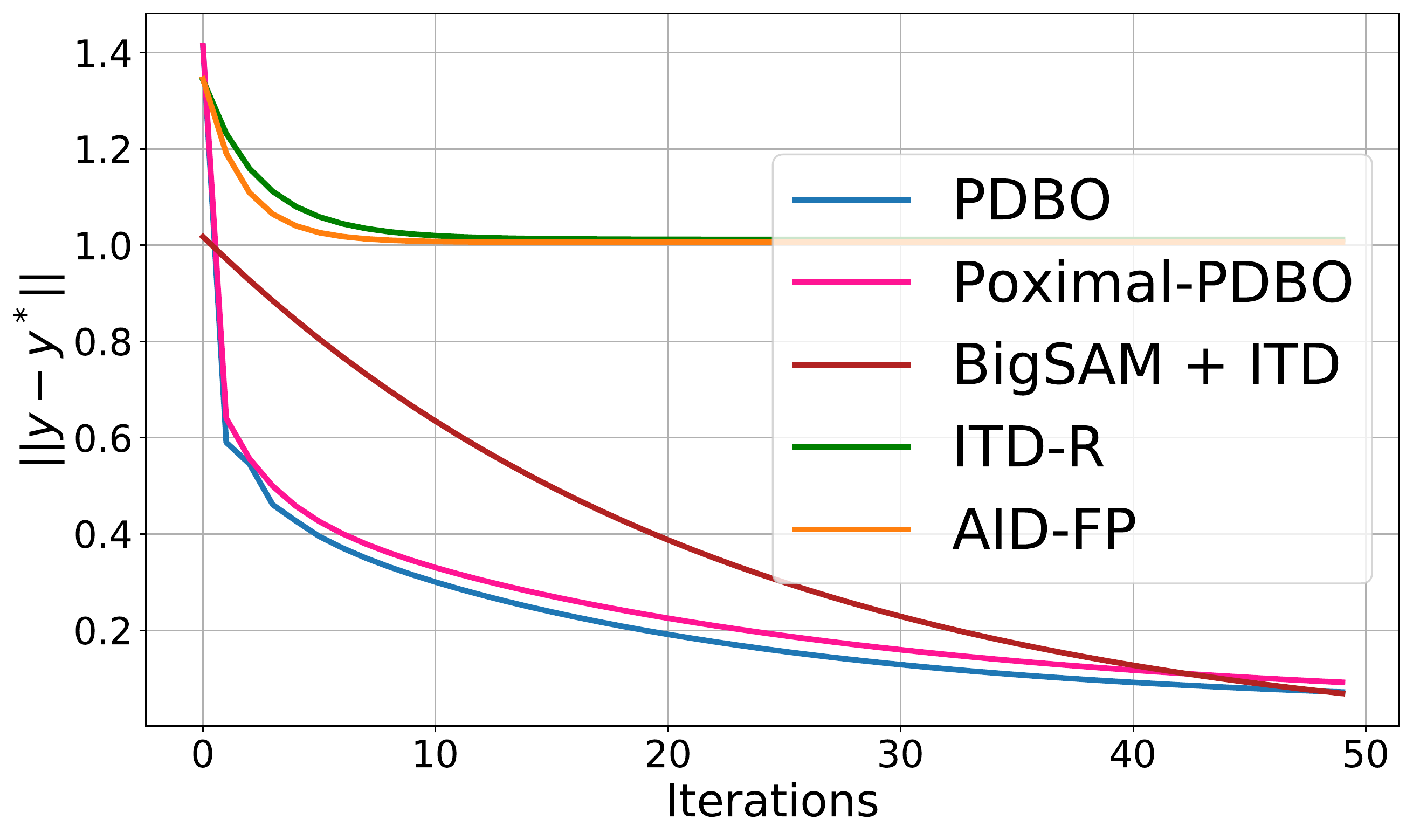}
			&\includegraphics[width=4.2cm,height=2.8cm]{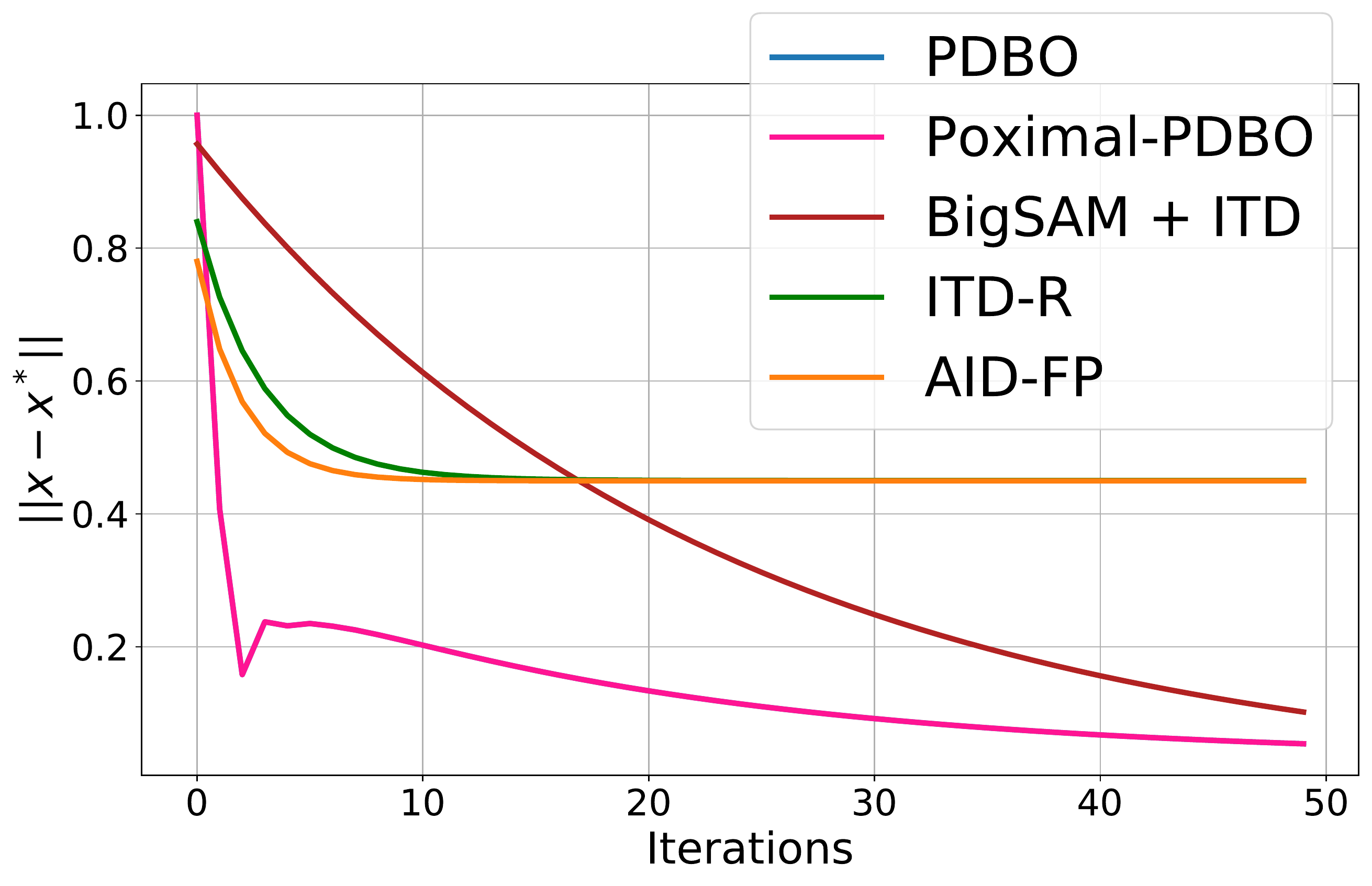} \\ 
		\end{tabular} 
		\caption{Evaluation of the compared algorithms using different optimality metrics. First row and first two plots in second row: $(x_0, y_0) = (2, (0.5, 0.5))$. Last plot in second row and third row: $(x_0, y_0) = (0, (2, 2))$. In both cases, PDBO is initialized with $\lambda=2$.} 
	\label{fig:toy} 
\end{figure*}
For our \textbf{PDBO}, we set the learning rates $\tau_t$, $\eta_t$ to be constants $0.1$, $0.2$ respectively, and $\theta_t=0$. For our \textbf{Proximal-PDBO}, we set the $\tau_t$, $\eta_t$ and $\theta_t$ to be the same as \textbf{PDBO}. Moreover, we specify $T=50$. For all compared methods, we fix the inner and outer learning rates to respectively $0.5$ and $0.2$. We use $N = 5$ gradient descent steps to estimate the minimimal value of the smoothed inner-objective and use the same number of iterations for all compared methods. 

\Cref{fig:toy} shows several evaluation metrics for the algorithms under comparison over different initialization points. It can be seen that our two algorithms \textbf{(Proximal-)PDBO} reach the optimal solution at the fastest rate. Also as analytically proved in \citep{liu2020generic}, several plots show that, with different initialization points, the classical AID and ITD methods cannot converge to the global optimal solution of the problem in \cref{eq:toy}. In particular, algorithms \textbf{AID-FP} and \textbf{ITD-R} are both stuck in a bad local minima of the problem (\cref{fig:toy} (c), (d), (e), (h), and (i)). This is essentially due to the very restrictive unique minimizer assumption of these methods. When this fundamental requirement is not met, such approaches are not equipped with mechanisms to select among the multiple minimizers. Instead, our algorithm, which solves a constrained optimization problem, leverages the constraint set to guide the optimization process. In the appendix, we have provide evidence that our \textbf{PDBO} do converge to a KKT point of the reformulated problem in \cref{eq:reform3} for the problem \cref{eq:toy}. 

{\bf Further Experiments.} Besides the example in \cref{eq:toy}, we conduct another experiment (presented in the appendix) to demonstrate that in practice \textbf{PDBO} can converge well for a broader class of functions even when $g(x, y)$ is not convex on $y$. The results show that  \textbf{PDBO} outperforms all state-of-the-art methods with a large margin.

\subsection{Hyperparameter Optimization}
The goal of hyperparameter optimization (HO) is to search for the set of hyperparameters that yield the optimal value of some model selection criterion (e.g., loss on unseen data). HO can be naturally expressed as a bilevel optimization problem, in which at the inner level one searches for the model parameters that achieve the lowest training loss for given hyperparameters. At the outer level, one optimizes the hyperparameters over a validation dataset. The problem can be mathematically formulated as follows 
\begin{equation*}
	\min_{\lambda, w\in\mathcal{W}(\lambda)} \mathcal{L}_{\mathrm{val}}(\lambda, w) \coloneqq \frac{1}{\left|\mathcal{D}_{\mathrm{val}}\right|} \sum_{\xi \in \mathcal{D}_{\mathrm{val}}} \mathcal{L}\left(\lambda, w ; \xi\right),\quad \mbox{with} \quad \mathcal{W}(\lambda) \coloneqq\argmin_{w} \mathcal{L}_{\mathrm{tr}}( \lambda, w),
\end{equation*} 
where 
	$\mathcal{L}_{\mathrm{tr}}( \lambda, w):=\frac{1}{\left|\mathcal{D}_{\mathrm{tr}}\right|} \sum_{\zeta \in \mathcal{D}_{\mathrm{tr}}}(\mathcal{L}(\lambda, w ; \zeta)+\mathcal{R}(\lambda, w))$, $\gL$ is a loss function, $\mathcal{R}(w, \lambda)$ is a regularizer, and $\gD_{\mathrm{tr}}$ and $\gD_{\mathrm{val}}$ are respectively training and validation data.  

Following \citep{franceschi2017forward, grazzi2020bo}, we perform classification on the 20 Newsgroup dataset, where the classifier is modeled by an affine transformation and the cost function $\gL$ is the cross-entrpy loss. 
We set one $\ell_2$-regularization hyperparameter for each weight in $w$, so that $\lambda$ and $w$ have the same size. 
For our algorithm \textbf{PDBO}, we optimize the parameters and hyperparameters using gradient descent with a fixed learning rate of $\eta_t^{-1} = 100$. We set the learning rate for the dual variable to be $\tau_t^{-1}=0.001$. We use $N = 5$ gradient descent steps to estimate the minimal value of the smoothed inner-objective. For \textbf{BigSAM+ITD}, we set the averaging parameter to $0.5$ and fix the inner and outer learning rates to be $100$. For \textbf{AID-FP} and \textbf{ITD-R}, we use the suggested parameters in their implementations accompanying the paper \citep{grazzi2020bo}. 
\begin{figure}[htbp] 
\centering 
\begin{tabular}{cc}
    \includegraphics[width=6cm,height=3.5cm]{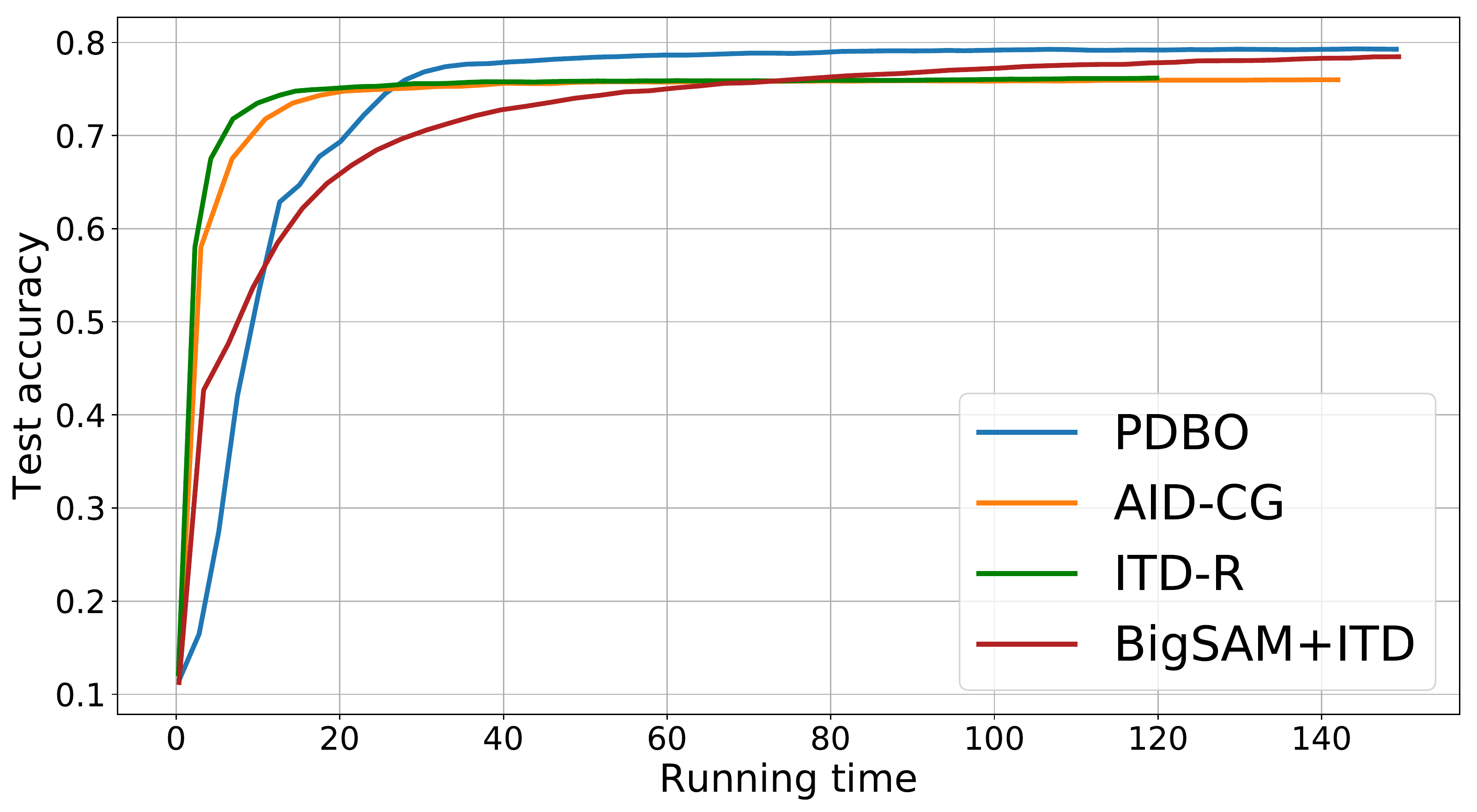} &\qquad \qquad\qquad\begin{minipage}{0.45\textwidth}{
        \vspace{-3.5cm}
        \begin{tabular}{c|c}
        Algo. & Acc. \\
        \hline \hline 
        PDBO& \bf 79.26 \\ 
        \hline
        AID-FP& 76.20 \\ 
        \hline
        ITD-R& 76.20 \\
        \hline
        BigSAM+ITD& 78.92 \\
        \hline
        \end{tabular}
        }
        \end{minipage}
\end{tabular} 
\caption{Classification results on 20 Newsgroup datase. {\bf Left plot:} accuracy on test data v.s. running time. {\bf Right table:} final best test accuracy.}  
\label{fig:20news} 
\end{figure}

The evaluations of the algorithms under comparison on a hold-out test dataset is shown in \Cref{fig:20news}. It can be seen that our algorithm {\bf PDBO} significantly improves over AID and ITD methods, and slightly outperforms \textbf{BigSAM+ITD} method with a much faster convergence speed. 
 
\section{Conclusion}

In this paper, we investigated a bilevel optimization problem where the inner-level function has multiple minima. Based on the reformulation of such a problem as an  constrained optimization problem, we designed two algorithms PDBO and Proximal-PDBO using primal-dual gradient descent and ascent method. Specifically, PDBO features a simple design and implementation, and Proximal-PDBO features a strong convergence guarantee that we can establish. We further conducted experiments to demonstrate the desirable performance of our algorithm. As future work, it is interesting to study bilevel problems with a nonconvex inner-level function, which can have multiple inner minima. While our current design can serve as a starting point, finding a good characterization of the set of inner minima can be challenging.

\bibliography{iclr2021_conference}
\bibliographystyle{iclr2021_conference}

\newpage
\appendix 
\textbf{\Large Supplementary Materials}
\section{Optimization Paths of PDBO}
 In \Cref{fig:toymore}, we plot the optimization paths of \textbf{PDBO} when solving the problem in \cref{eq:toy}. 
The two plots show that the optimization terminates with a strictly positive dual variable and that the constraint is satisfied with equality. Thus, the slackness condition is satisfied. This also confirms that our algorithm \textbf{PDBO} did converge to a KKT point of the reformulated problem.
\begin{figure}[h]
    \centering
	\begin{tabular}{cc}
		\small{(a) dual variable v.s. iterations}  & \small{(b) constraint v.s. iterations} \\
		\includegraphics[width=6cm,height=4cm]{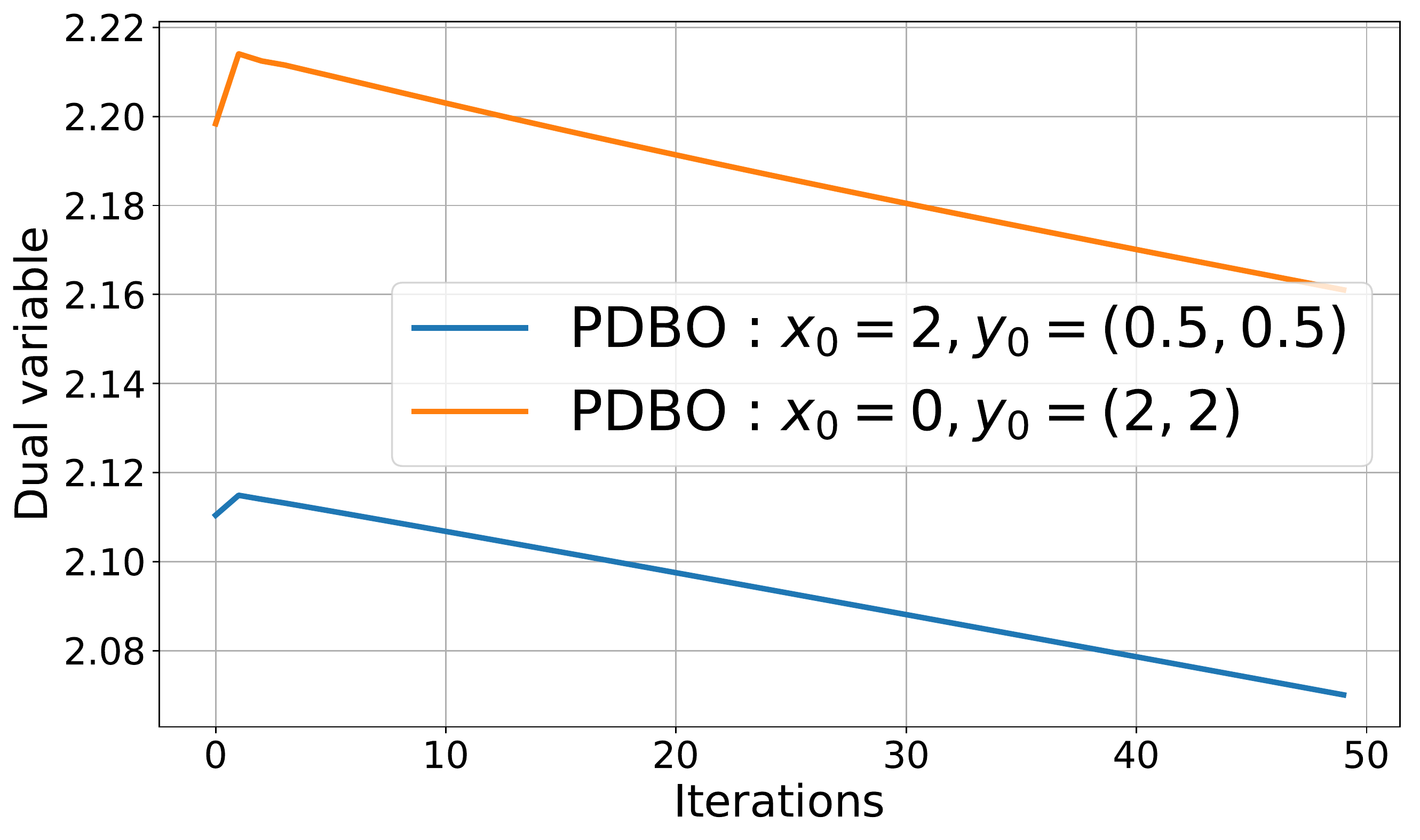}
		&\includegraphics[width=6cm,height=4cm]{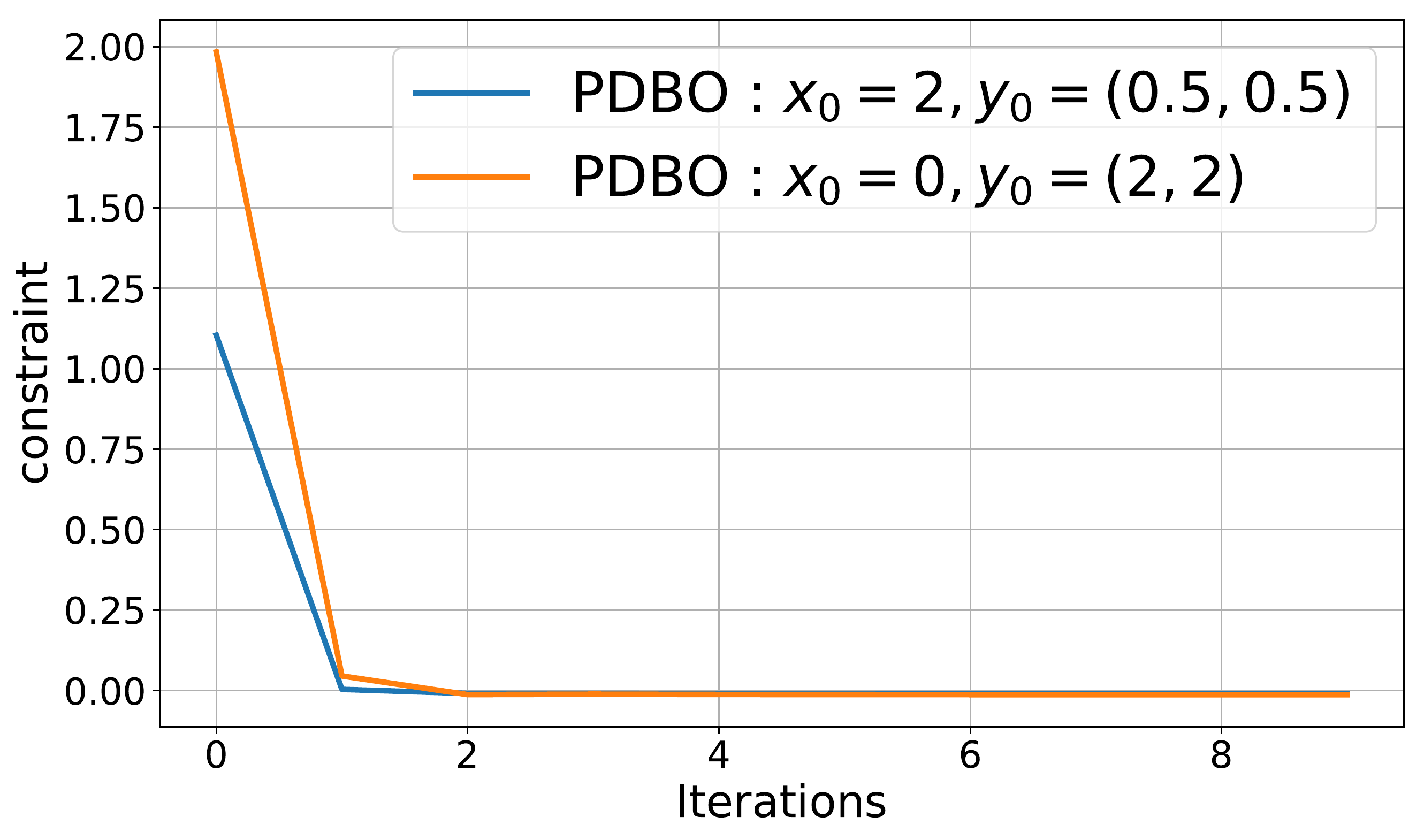}
	\end{tabular}
	\vspace{-3mm}
	\caption{Optimization path of dual variable and constraint values for different initializations.}
\label{fig:toymore}
\vspace{-3mm}
\end{figure} 

\section{Additional Experiment with Multiple Minima}

In this section, we demonstrate that \textbf{(Proximal-)PDBO} is applicable to a more general class of bilevel problems in practice, where the inner problem is not necessarily convex on $y$ (as we require in the theoretical analysis), but still have multiple minimal points.  

Consider the following bilevel optimization problem
 \begin{align}
     \min _{x \in \mathcal{C}, y \in \mathcal{S}_x} \|x-a\|^{2}+\|y-a\|^{2},\quad \mbox{where}\quad \mathcal{S}_x \coloneqq \argmin_{y\in \mathcal{C}} \sin (x+y),\label{eq:toy2}
 \end{align}
 where $\mathcal{C} = [-10, 10]$, and $a$ is a constant (we set $a=0$ in our experiment). 

Clearly, the problem in \cref{eq:toy2} does not degenerate to single inner minimum bilevel optimization due to the sinusoid in the lower function. Such a problem is harder than the one in \cref{eq:toy}, and it violates the assumption that $g(x,y)$ is convex on $y$.

In this experiment, we compare \textbf{PDBO}, \textbf{Proximal-PDBO}, \textbf{BigSAM+ITD}, \textbf{ITD-R} and \textbf{AID-FP}. We initialize all methods with $(x,y) = (3,3)$ and the hyperparameters are set to be the same as what we did in \Cref{sec:experiment}. The results are provided in \Cref{fig:newtoyes}. The plots show that \textbf{ITD-R} and \textbf{AID-FP} methods converge to a bad stationary point and could not find good solutions. The methods \textbf{PDBO}, \textbf{Proximal-PDBO}, and \textbf{BigSAM+ITD} converge to better points but our algorithms \textbf{PDBO} and \textbf{Proximal-PDBO} significantly outperform \textbf{BigSAM+ITD}. 
\vspace{-0.4cm}
\begin{figure}[!h]
\centering
\begin{tabular}{ccc}
    (a) outer optimality gap & (b) outer convergence error  
     & (c) inner convergence error\\
    \includegraphics[width=4.5cm]{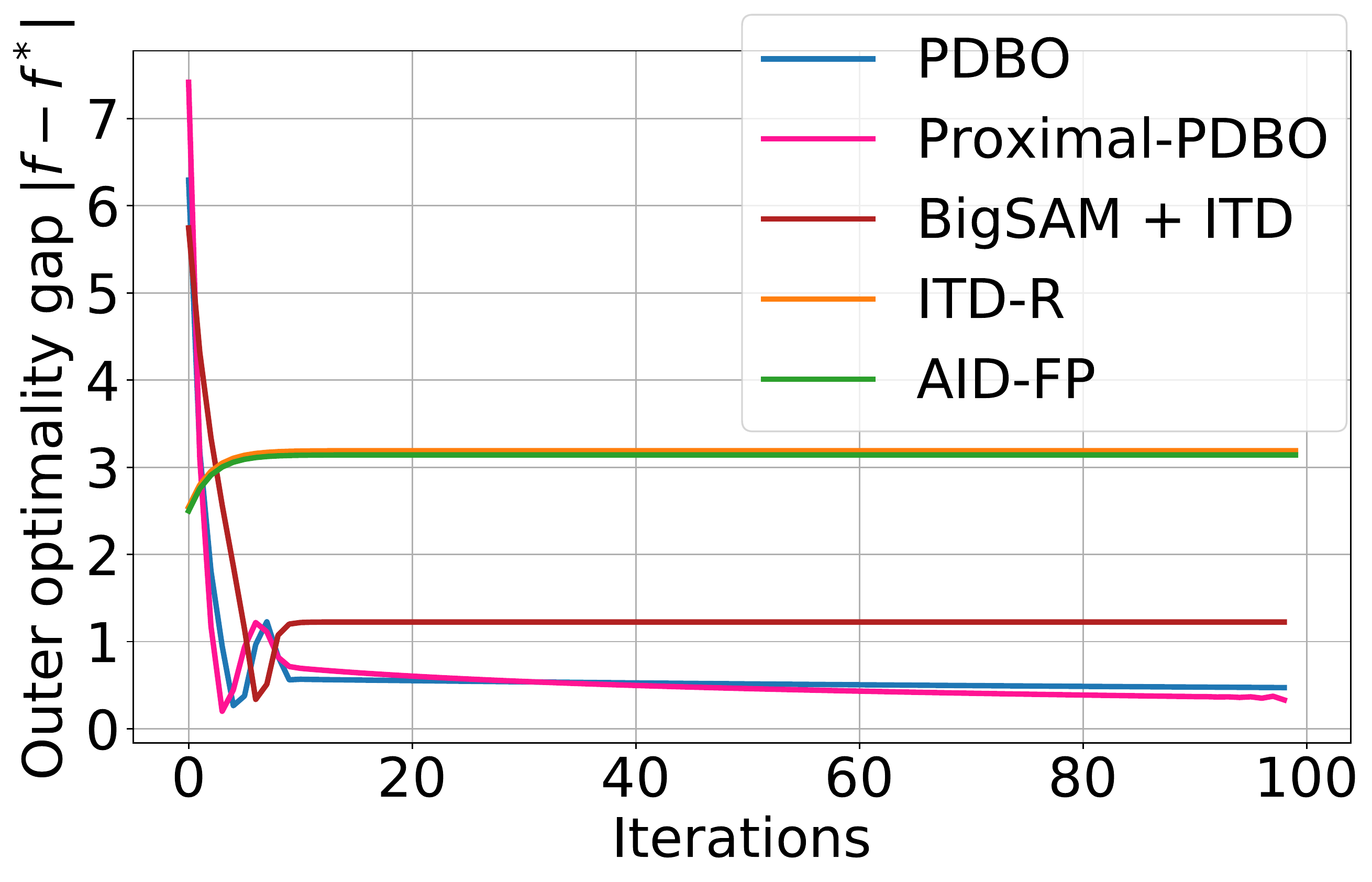} & \includegraphics[width=4.5cm]{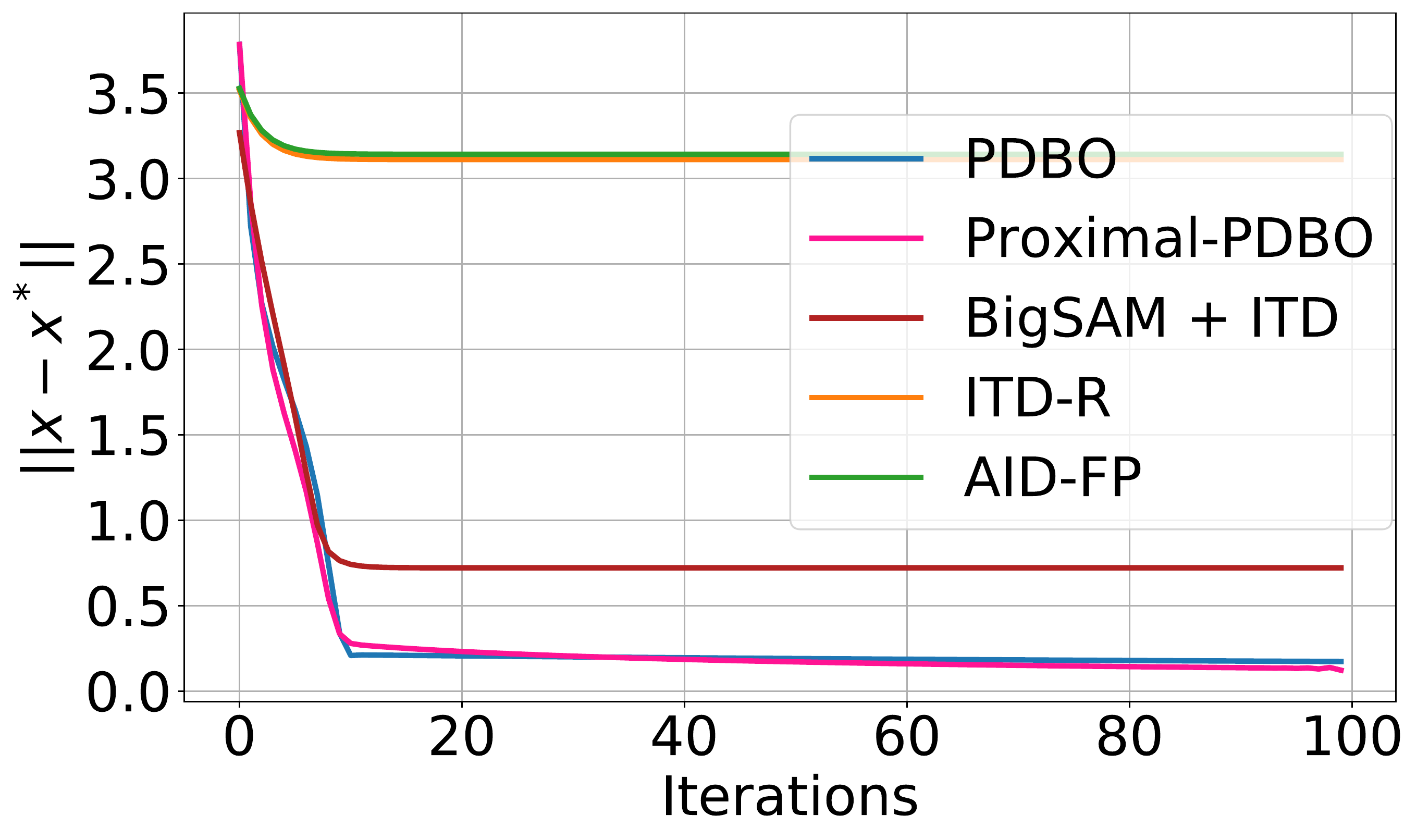} &\includegraphics[width=4.5cm]{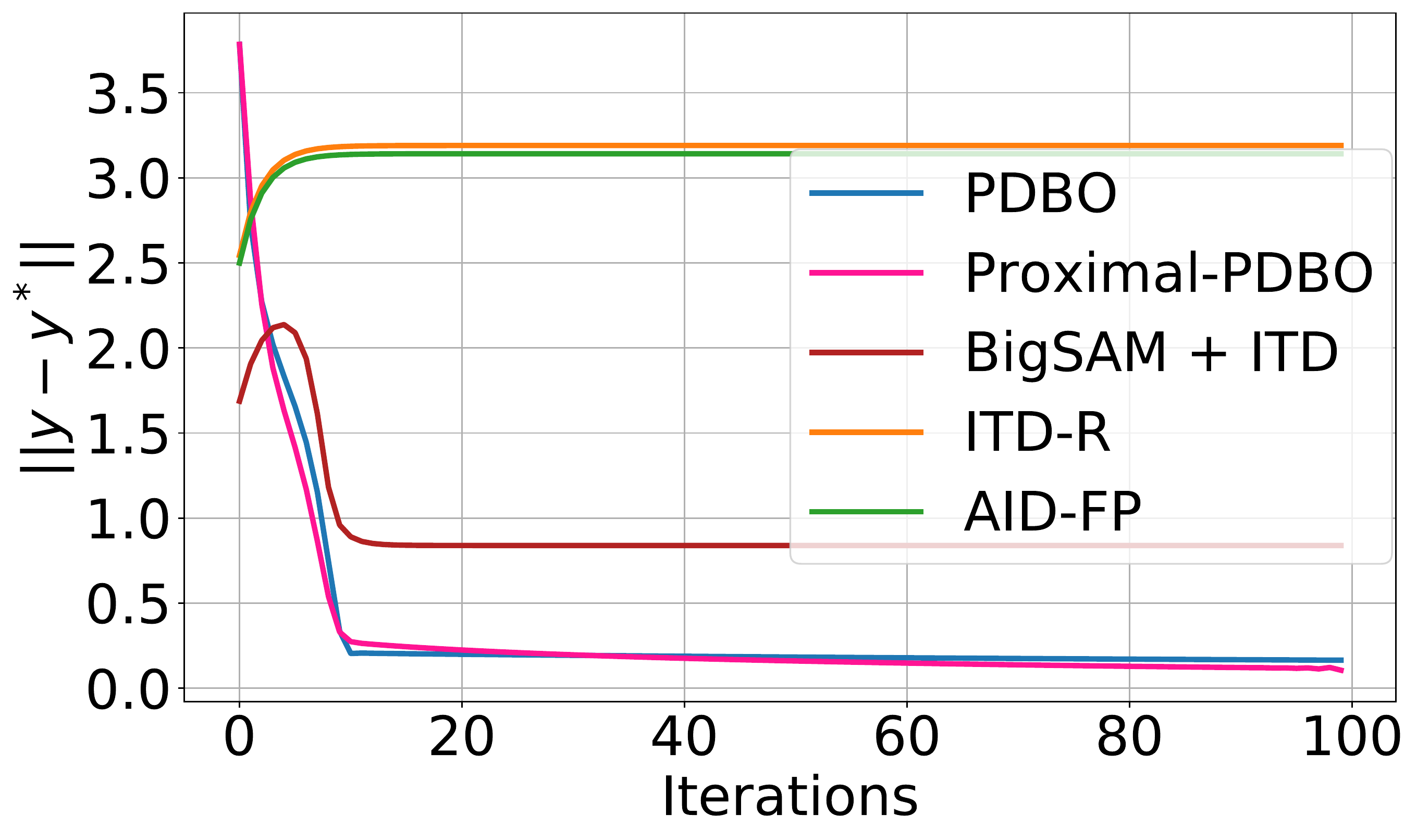} 
\end{tabular}
    \caption{Comparison of different algorithms}\label{fig:newtoyes}
\end{figure}  

\section{Calculation of $\nabla \tilde{h}(z)$ in \Cref{sec:formulation}}
For completeness, we provide the steps for obtaining the form of $\nabla \tilde{h}(z)$ in \Cref{sec:formulation}. For the ease of reading, we restate the result here. 
Suppose that \Cref{ass:smoothness} holds. Then the gradients $\nabla_x \tilde{h}(x,y)$ and $\nabla_y \tilde{h}(x,y)$ of function $\tilde{h}(x, y)$ take the following forms: 
\begin{align}\label{gradhx}
\nabla_x \tilde{h}(x,y) &= \nabla_x g(x,y) - \nabla_x g(x,\tilde{y}^*(x)), \\ 
\nabla_y \tilde{h}(x,y) &= \nabla_y g(x,y). \label{gradhy}
\end{align} 
\begin{proof}
First, \cref{gradhy} follows immediately. Hence, we prove only \cref{gradhx}. 
Recall the definition of $\tilde{h}(x, y)$:
\begin{align*}
	\tilde{h}(x, y) = g(x,y) - \tilde{g}^*(x) - \delta, 
\end{align*}
where $\tilde{g}^*(x) = \tilde{g}(x,\tilde{y}^*(x))$ with $\tilde{y}^*(x) = \argmin_y \tilde{g}(x,y) = g(x,y) + \frac{\alpha}{2}\|y\|^2.$
Using the chain rule to compute the gradient with respect to $x$ of function $\tilde{h}(x, y) = g(x,y) - g(x,\tilde{y}^*(x)) - \frac{\alpha}{2}\|\tilde{y}^*(x)\|^2 - \delta$ yields: 
\begin{align}
	\nabla_x \tilde{h}(x,y) &= \nabla_x g(x,y) - \big[\nabla_x g(x,\tilde{y}^*(x)) +  \frac{\partial \tilde{y}^*(x)}{\partial x} \nabla_y g(x,\tilde{y}^*(x))\big] - \alpha \frac{\partial \tilde{y}^*(x)}{\partial x}\tilde{y}^*(x) \nonumber\\ 
	& = \nabla_x g(x,y) - \nabla_x g(x,\tilde{y}^*(x)) - \frac{\partial \tilde{y}^*(x)}{\partial x} \big[\nabla_y g(x,\tilde{y}^*(x)) + \alpha \tilde{y}^*(x)\big]. \nonumber
\end{align}
The first order optimality condition ensures that $\nabla_y g(x,\tilde{y}^*(x)) + \alpha \tilde{y}^*(x) = 0$. Hence, we obain the desired result: 
\begin{align}
	\nabla_x \tilde{h}(x,y) &= \nabla_x g(x,y) - \nabla_x g(x,\tilde{y}^*(x)). \nonumber
\end{align}
\end{proof}

\section{Proof of \Cref{thm:pdboconvergence}} \label{sec:proofpdboconvergence}
\subsection{Supporting Lemmas}
We first cite two standard lemmas, which are useful for our proof here.
\begin{lemma}[Lemma 3.5 \cite{lan2020first}]\label{lemma:threepoint}
	Suppose that $\mathcal{S}$ is a convex and closed subset of $\mathbb{R}^n$, $x\in\mathcal{S}$, and $v\in\mathbb{R}^n$. Define 
	$\bar{x} = \Pi_{\mathcal{S}} \left(x - v\right)$.
	Then, for any $\tilde{x} \in\mathcal{S}$, the following inequality holds:
	\begin{align*}
		\langle x,v\rangle+ \tfrac{1}{2}\|\bar{x} -\tilde{x}\|_2^2 + \tfrac{1}{2}\|x -\bar{x}\|_2^2 \le \tfrac{1}{2}\|x -\tilde{x}\|_2^2.
	\end{align*}
\end{lemma}
\begin{lemma}[Theorem 2.2.14 \cite{nesterov2018lectures}]\label{lemma:pgd}
	Suppose that \Cref{ass:smoothness} holds. Consider the projected gradient descent in \cref{eq:projectedgradientdescent}. Define $\tilde{y}^*(x_t) \coloneqq \argmin_{y\in\mathcal{Y}} g(x_t, y) + \tfrac{\alpha}{2}\|y\|_2^2$. We have 
	\begin{align*}
		\|\hat{y}^*(x_t) - \tilde{y}^*(x_t)\|_2 = \|\hat{y}_N - \tilde{y}^*(x_t)\|_2 \le \left(1 - \tfrac{\alpha}{\rho_g + 2\alpha}\right)^N \|\hat{y}_0 - \tilde{y}^*(x_t)\|_2.
	\end{align*}
\end{lemma}

Next, we establish an upper bound on the optimal dual variable in the following lemma.
\begin{lemma}
Suppose that \Cref{ass:smoothness,ass:strongassp} hold. Then, there exists $\lambda^*$ satisfying $0\le \lambda^*\le \tfrac{D_f}{\delta}$, where $D_f = \sup_{z, z^\prime\in\mathcal{Z}} |f(z) - f(z^\prime)|$, so that for
$z^* := \text{argmin}_{z\in\mathcal{Z}} f(z) + \lambda^* \tilde{h}(z)$, $\tilde{h}(z^*)\le 0$, and the KKT condition holds, i.e.,  $\lambda^*\tilde{h}(z^*) = 0$, and $\nabla f(z^*) + \lambda^*\nabla \tilde{h}(z^*) \in -\mathcal{N}_{\mathcal{Z}}(z^*)$, where $\mathcal{N}_{\mathcal{Z}}(z^*)$ is the normal cone defined as $\mathcal{N}_{\mathcal{Z}}(z^*) = \{ v \in \mathbb{R}^{p+d}: \langle v, z- z^*\rangle\le 0, \mbox{ for all } z \in\mathcal{Z}\}$.  
\end{lemma}
\begin{proof}
Pick any $x_0\in\mathcal{X}$. Let $y_0 = \argmin_{y\in\mathcal{Y}} g(x,y) + \tfrac{\alpha}{2}\|y\|_2^2$, and $z_0 = (x_0, y_0)$. Then, $\tilde{h}(z_0) = g(x_0, y_0) - \tilde{g}^*(x_0) -\delta = -\delta$ holds, which implies that $z_0$ is a strictly feasible point. The existence of such a strictly feasible point $z_0$ ensures that the Slater's condition holds, and then the standard result (see, e.g., \cite{lan2020first}) implies the existence of $\lambda^*\ge 0$ that satisfies for
$z^* := \text{argmin}_{z\in\mathcal{Z}} f(z) + \lambda^* \tilde{h}(z)$, $\tilde{h}(z^*)\le 0$,  $\lambda^*\tilde{h}(z^*) = 0$, and $\nabla f(z^*) + \lambda^*\nabla \tilde{h}(z^*) \in -\mathcal{N}_{\mathcal{Z}}(z^*)$.  

Define the dual function as $d(\lambda): = \min_{z\in\mathcal{Z}} \mathcal{L}(z, \lambda)$. Then, we have, for any $\lambda$ and $z\in\mathcal{Z}$, 
\begin{align}\label{eq:dualfunctionbound}
	d(\lambda) \le  f(z_0) + \lambda\hat h(z_0) = f(z_0)  -\delta\lambda.
\end{align}
Taking $\lambda=\lambda^*$ in \cref{eq:dualfunctionbound} and using the fact that $|d(\lambda^*)- f(z_0)| = | f(z^*)- f(z_0)|\le D_f$, where $D_f= \sup_{z, z^\prime\in\mathcal{Z}} |f(z) - f(z^\prime)|$, we complete the proof.
\end{proof}
In the next lemma, we show that the constrained function $\tilde{h}(z)$ is gradient Lipschitz continuous.
\begin{lemma}\label{lemma:lipschitzh}
Suppose that \Cref{ass:smoothness} holds. Then, the gradient $\nabla \tilde{h}(z) = \big(\nabla_x \tilde{h}(x,y), \nabla_y \tilde{h}(x,y)\big)^\top$ is Lipschitz continuous with constant $\rho_h = \rho_g (2+\rho_g/\alpha)$. 
\end{lemma}
\begin{proof}
Recall the form of $\nabla \tilde{h}(z)$ is given by
\begin{align}
    \nabla \tilde{h}(z) =& \left(\begin{array}{c} \nabla_x g(x,y)\\
    \nabla_y g(x,y)
\end{array}\right) - \left(\begin{array}{c}\nabla_x g(x,\tilde{y}^*(x))\\
    \mathbf{0}_d
\end{array}\right)
    = \nabla g(z) - \left(\begin{array}{c} G_x\\
    G_y
\end{array}\right),\nonumber
\end{align}
where $G_x: = \nabla_x g(x,\tilde{y}^*(x))$ and $G_y := \mathbf{0}_d \in \mathbb{R}^d$ is a vector of all zeros. Taking derivative w.r.t. $z$ yields: 
\begin{align}
    \nabla^2 \tilde{h}(z) =& \nabla^2 g(z) - \left(\begin{array}{cc} \frac{\partial G_x}{\partial x} & \frac{\partial G_x}{\partial y} \\ 
    \frac{\partial G_y}{\partial x} & \frac{\partial G_y}{\partial y} 
\end{array}\right) \nonumber \\
    =& \nabla^2 g(z) - \underbrace{\left(\begin{array}{cc} \frac{\partial \nabla_x g(x,\tilde{y}^*(x))}{\partial x} & \mathbf{0}_{p\times d} \\ 
    \mathbf{0}_{d\times p}   & \mathbf{0}_{d\times d}  
\end{array}\right)}_{M} \label{gradh}.
\end{align}
where $\mathbf{0}_{m\times n} \in \mathbb{R}^{m \times n}$ is a matrix of all zeros. 

Recall that in \Cref{ass:smoothness}, we assume $\|\nabla g(z) - \nabla g(z^\prime)\|_2\le \rho_g\|z - z^\prime\|_2$ for all $z,z^\prime\in\mathcal{Z}$, which immediately implies $\|\nabla^2 g(z)\|_2\le \rho_g$. Note that in the sequel, $\|\cdot\|_2$ of a matrix denotes the spectral norm. Moreover, let \[\nabla^2 g(z) = \begin{pmatrix}\nabla^2_{xx} g(z) & \nabla^2_{xy} g(z)\\ \nabla^2_{yx} g(z) & \nabla^2_{yy} g(z)\end{pmatrix}.\]
Given any $x\in\mathbb{R}^p$, we have  $\nabla^2 g(z) \begin{pmatrix}x\\0\end{pmatrix}= \begin{pmatrix}\nabla^2_{xx} g(z)x\\\nabla^2_{yx} g(z)x\end{pmatrix}$. Thus, the following inequality holds
\begin{align}
    \|\nabla^2_{xx} g(z) x \|_2  \le \left\|\begin{pmatrix}\nabla^2_{xx} g(z)x\\\nabla^2_{yx} g(z)x\end{pmatrix}\right\|_2 = \left\|\nabla^2 g(z) \begin{pmatrix}x\\0\end{pmatrix}\right\|_2 \le \|\nabla^2 g(z)\|_2 \|x\|_2\le \rho_g\|x\|_2.\label{eq:middle60}
\end{align}
Following the definition of the spectral norm of $\nabla_{xx}^2 g(z)$, we have 
\begin{equation}
    \|\nabla_{xx}^2 g(z)\|_2\le \rho_g.\label{eq:l2normxx}
\end{equation}
Following the similar steps to \cref{eq:middle60}, we have 
\begin{equation}
    \|\nabla_{yx}^2 g(z)\|_2\le \rho_g.\label{eq:l2normyx}
\end{equation}
We next upper-bound the spectral norm of the matrix $M$ defined in \cref{gradh}. We have: 
\begin{align}\label{eq:normM}
    \|\nabla^2 \tilde{g}^*(z)\|_2 = \left\|M\right\|_2 &= \left\| \left(\begin{array}{cc} \frac{\partial \nabla_x g(x,\tilde{y}^*(x))}{\partial x} & \mathbf{0}_{p\times d}  \\ 
    \mathbf{0}_{d\times p}   & \mathbf{0}_{d\times d}  
\end{array}\right) \right\|_2 \leq \left\|\frac{\partial \nabla_x g(x,\tilde{y}^*(x))}{\partial x}\right\|_2,
\end{align}
where \cref{eq:normM} follows from the fact that for the block matrix $C = \left(\begin{array}{cc} A & \mathbf{0}  \\ 
    \mathbf{0} & B \end{array}\right)$, we have $\big\| C \big \|_2 \leq \max \big\{\big\| A \big \|_2, \big\| B \big \|_2\big\}$. Further, using the chain rule, we obtain
\begin{align}
    \frac{\partial \nabla_x g(x,\tilde{y}^*(x))}{\partial x} = \nabla_x^2 g(x,\tilde{y}^*(x)) +  \frac{\partial \tilde{y}^*(x)}{\partial x} \nabla_y \nabla_x g(x,\tilde{y}^*(x)).\label{eq:nablagstar}
\end{align}
Thus, taking the norm on both sides of the above equation and applying the triangle inequality yield
\begin{align}
    \Big\|\frac{\partial \nabla_x g(x,\tilde{y}^*(x))}{\partial x}\Big\|_2 &\leq  \Big\|\nabla_{xx}^2 g(x,\tilde{y}^*(x))\Big\|_2 +  \Big\|\frac{\partial \tilde{y}^*(x)}{\partial x}\Big\|_2 \Big\|\nabla_y \nabla_x g(x,\tilde{y}^*(x))\Big\|_2 \nonumber \\
    &\overset{(i)}\leq \rho_g +  \rho_g \left\|\frac{\partial \tilde{y}^*(x)}{\partial x}\right\|_2, \label{partgradx2}
\end{align}
where $(i)$ follows from \cref{eq:l2normxx,eq:l2normyx}.

Applying implicit differentiation w.r.t. $x$ to the optimality condition of $\tilde{y}^*(x)$ implies $\nabla_y g(x,\tilde{y}^*(x)) + \alpha \tilde{y}^*(x) = 0$. This yields
\begin{align*}
    \nabla_x \nabla_y g(x,\tilde{y}^*(x)) + \frac{\partial \tilde{y}^*(x)}{\partial x} \nabla_y^2 g(x,\tilde{y}^*(x)) + \alpha \frac{\partial \tilde{y}^*(x)}{\partial x} = 0,
\end{align*}
which further yields
\begin{align*}
    \frac{\partial \tilde{y}^*(x)}{\partial x} = - \Big[ \nabla_y^2 g(x,\tilde{y}^*(x)) + \alpha \mathbf{I}\Big]^{-1} \nabla_x \nabla_y g(x,\tilde{y}^*(x)).
\end{align*}
Hence, we obtain
\begin{align}
    \Big\|\frac{\partial \tilde{y}^*(x)}{\partial x}\Big\|_2 \leq & \Big\| \Big[\nabla_y^2 g(x,\tilde{y}^*(x)) + \alpha \mathbf{I}\Big]^{-1} \Big\|_2 \Big\| \nabla_x \nabla_y g(x,\tilde{y}^*(x)) \Big\|_2 \leq \frac{\rho_g}{\alpha},  \label{partypartx2}
\end{align}
where the last inequality follows from \Cref{ass:smoothness}. Hence, combining \cref{eq:normM}, \cref{partgradx2}, and \cref{partypartx2}, we have 
\begin{align}
    \|\nabla^2 \tilde{g}^*(z)\|_2= \big\|M\big\|_2 \leq \rho_g +  \rho_g \frac{\rho_g}{\alpha}, \label{eq:nablagstar2}
\end{align}
which, in conjunction of \cref{gradh}, yields 
\begin{align}
    \big \|\nabla^2 \tilde{h}(z)\big\|_2 \leq \big \|\nabla^2 g(z)\big\|_2 + \big\|M\big\|_2 \leq \rho_g + \rho_g + \frac{\rho_g^2}{\alpha} = \rho_g \left(2 + \frac{\rho_g}{\alpha} \right).\nonumber
\end{align}
This completes the proof. 
\end{proof}

\subsection{Proof of \Cref{thm:pdboconvergence}}
Based on the above lemmas, we develop the proof of \Cref{thm:pdboconvergence}. We first formally restate the theorem with the full details.
\begin{theorem}
[Formal Statement of \Cref{thm:pdboconvergence}]\label{thm:pdboconvergencere}
Suppose \Cref{ass:smoothness} holds. Let $\gamma_t = t+t_0+1$, $\eta_t = \frac{\mu(t+t_0+1)}{2}$, $\tau_t = \frac{4L_g^2}{\mu t}$, $\theta_t = \frac{t+t_0}{t+t_0+1}$, where  $L_g= \sup_{z}\|\nabla_z g(z)\|_2$, $t_0 = \frac{2(\rho_f+ B\rho_h)}{\mu}$, $B = \tfrac{D_f}{\delta}+1$, where $D_f = \sup_{z, z^\prime\in\mathcal{Z}}|f(z) - f(z^\prime)|$, and $\rho_h$ is given in \Cref{lemma:lipschitzh}. Then, we have 
\begin{align*} 
	f(\bar{z}) - f( z^*)&\le   \frac{2L_gBD_\mathcal{Z}}{T^2} + \frac{\gamma_0(\eta_0 -\mu)\|{z}^*- z_0\|_2^2}{T^2}+ (\rho_gD_\mathcal{Z} + 4L_g)BD_\mathcal{Z} \left(1 - \tfrac{\alpha}{\rho_g + 2\alpha}\right)^N,
\end{align*}
\begin{align*}
    	[\tilde{h}(\bar{z})]_+\le \frac{2L_gBD_\mathcal{Z}+ \gamma_0\tau_0 B^2+\gamma_0(\eta_0 -\mu){D}_\mathcal{Z}^2}{T^2} +  (\rho_gD_\mathcal{Z} + 4L_g)BD_\mathcal{Z}\left(1 - \tfrac{\alpha}{\rho_g + 2\alpha}\right)^N,
\end{align*}
and
\begin{align*}
    \|\bar{z} - z^*\|_2^2 \le \frac{2{\gamma_0\tau_0} B^2 + 2{\gamma_0(\eta_0 -\mu)}D_\mathcal{Z}^2}{\mu T^2} +  \frac{2(T+t_0+1)^2}{T^2}(\rho_gD_\mathcal{Z} + 4L_g)BD_\mathcal{Z}\left(1 - \tfrac{\alpha}{\rho_g + 2\alpha}\right)^N,
\end{align*}
where $z^* = \argmin_{z\in\mathcal{Z}}\{f(z): \tilde h(z)\le 0\}$.
\end{theorem}

\begin{proof} 
We first define some notations that will be used later. Let $\hat{d}_t = (1 +\theta_t)\hat{h}(z_t) - \theta_t\hat{h}(z_{t-1})$, ${d}_t = (1+\theta_t)\tilde{h}(z_t) - \theta_t\tilde{h}(z_{t-1})$, and $\xi_t = \hat{h}(z_t) - \hat{h}(z_{t-1})$. Furthermore, we define the primal-dual gap function as
\begin{align*}
	Q(w, \tilde{w}) \coloneqq f(z) + \tilde{\lambda} \tilde{h}(z) - \left(f(\tilde{z})+\lambda \tilde{h}(\tilde{z})\right),
\end{align*} 
where $w = (z, \lambda)$, $\tilde{w} =(\tilde{z}, \tilde{\lambda})\in\mathcal{Z}\times{\Lambda}$ are primal-dual pairs. 

Consider the update of $\lambda$ in \cref{eq:lambdaupdate1}. Applying \Cref{lemma:threepoint} with $v = -{\hat{d}_t}/{\tau_t}$, $\mathcal{S}= \Lambda$, $\bar{x}= \lambda_{t+1}$, $x = \lambda_{t}$ and letting $\tilde{x} = \lambda$ be an arbitrary point inside $\Lambda$, we have
\begin{align}
	-(\lambda_{t+1} - \lambda)\hat{d}_t \le \frac{\tau_t}{2} \left((\lambda -\lambda_t)^2 - (\lambda_{t+1} -\lambda_t)^2 - (\lambda - \lambda_{t+1})^2\right). \label{eq:lambdaupdatehere}
\end{align}

Similarly, consider the update of $z$ in \cref{eq:zupdate1}. Applying \Cref{lemma:threepoint} with \[v =  \tfrac{1}{\eta_t}\left(\nabla f(z_t) + \lambda_{t+1}\hat{\nabla}\tilde{h}(z_t)\right)\coloneqq\tfrac{1}{\eta_t}\hat\nabla \mathcal{L}(z_t,\lambda_{t+1}),\] $\mathcal{S}= \mathcal{Z}$, $\bar{x}=z_{t+1}$, $x = z_{t}$ and let $\tilde{x} =z$ be an arbitrary point inside $\mathcal{Z}$, we obtain
\begin{align}
	\langle\hat{\nabla}_z\mathcal{L}(z_t, \lambda_{t+1}), z_{t+1} - z\rangle \le \frac{\eta_t}{2} \left((z -z_t)^2 - (z_{t+1} -z_t)^2 - (z - z_{t+1})^2\right). \label{eq:zupdatehere}
\end{align}
Recall that $f(z)$ and $\tilde{h}(z)$ are $\rho_f$- and $\rho_h$-gradient Lipschitz (see \Cref{ass:smoothness} and \Cref{lemma:lipschitzh}). This implies
\begin{align}
	\langle \nabla f(z_t) , z_{t+1} -z_t\rangle \ge f(z_{t+1}) - f(z_t) -\frac{\rho_f\|z_t -z_{t+1}\|_2^2}{2},\label{eq:lipschitzf}\\
	\langle \nabla \tilde{h}(z_t) , z_{t+1} -z_t\rangle \ge \tilde{h}(z_{t+1}) - \tilde{h}(z_t) -\frac{\rho_h\|z_t -z_{t+1}\|_2^2}{2}.\label{eq:lipschitzg}
\end{align}
Moreover,  \Cref{ass:strongassp} assumes $f (z)$ is a $\mu$-strongly convex function, which yields
\begin{align}
	\langle \nabla f (z_t) , z_{t} -z\rangle \ge f (z_{t}) - f (z) +\frac{\mu\|z -z_{t}\|_2^2}{2}.\label{eq:stronglyconvexfn} 
\end{align}
The convexity of $\tilde h(z)$ in \Cref{ass:strongassp} ensures that
\begin{align}
	\langle \nabla \tilde{h} (z_t) , z_{t} -z\rangle \ge \tilde{h} (z_{t}) - \tilde{h} (z) .\label{eq:stronglyconvexhn}
\end{align}
For the exact gradient of Lagrangian with respect to the primal variable, we have 
\begin{align}
	&\langle {\nabla}_z\mathcal{L} (z_t, \lambda_{t+1}), z_{t+1} - z\rangle \nonumber\\
	&\quad= \langle \nabla f  (z_t) +\lambda_{t+1} \nabla \tilde{h} (z_{t}), z_{t+1} - z \rangle \nonumber \nonumber\\
	&\quad=  \langle \nabla f  (z_t), z_{t+1} - z_t\rangle + \langle\nabla f  (z_t), z_t - z\rangle +\lambda_{t+1} \langle \nabla \tilde{h}  (z_t), z_{t+1} - z_t\rangle + \lambda_{t+1}\langle\nabla \tilde{h}  (z_t), z_t - z\rangle \nonumber \\
	&\quad\overset{(i)}\ge  f(z_{t+1}) - f (z) + \lambda_{t+1} (\tilde{h} (z_{t+1}) -\tilde{h} (z))  - \frac{\rho_f + \lambda_{t+1}\rho_h\|z_{t+1} - z_t\|_2^2}{2} + \frac{\mu\|z-z_t\|_2^2}{2},\label{eq:exactzupdatehere}
\end{align}
where $(i)$ follows from \cref{eq:lipschitzf,eq:lipschitzg,eq:stronglyconvexfn,eq:stronglyconvexhn}.

Combining \cref{eq:zupdatehere,eq:exactzupdatehere} yields
\begin{align}
	f (z_{t+1}) - f (z) &\le \langle {\nabla}_z\mathcal{L} (z_t, \lambda_{t+1}) -  \hat{\nabla}_z\mathcal{L} (z_t, \lambda_{t+1}), z_{t+1} - z\rangle + \lambda_{t+1}(\tilde{h} (z) - \tilde{h} (z_{t+1}))\nonumber\\
	&\qquad  + \frac{\eta_t- \mu}{2}\|z - z_t\|_2^2 - \frac{\eta_t- (\rho_f + \lambda_{t+1}\rho_h)}{2}\|z_{t+1} - z_t\|_2^2 - \frac{\eta_t}{ 2}\|z - z_{t+1}\|_2^2. \label{eq:fwrtzhere}
\end{align}

Recall the definition of $\xi_t = \hat{h} (z_t) - \hat{h} (z_{t-1})$. Substituting it into \cref{eq:lambdaupdatehere} yields
\begin{align} 
	0 \le& -(\lambda - \lambda_{t+1})\hat{h}  (z_{t+1}) -(\lambda_{t+1} -\lambda)\xi_{t+1} + \theta_t(\lambda_{t+1} -\lambda)\xi_{t}\nonumber\\ 
	&\quad + \frac{\tau_t}{2}\left((\lambda -\lambda_t)^2 - (\lambda_{t+1} -\lambda_t)^2 - (\lambda - \lambda_{t+1})^2\right).\label{eq:fwrtlambdahere}
\end{align}

Let $w= (z, \lambda)$ and $w_{t+1} = (z_{t+1}, \lambda_{t+1})$. By the definition of the primal-dual gap function, we have 
\begin{align}
	&Q(w_{t+1}, w)\nonumber \\
	&\quad= f (z_{t+1}) + \lambda \tilde{h} (z_{t+1}) -f (z) -\lambda_{t+1}\tilde{h} (z)\nonumber\\
	&\quad\overset{(i)}\le  \langle {\nabla}_z\mathcal{L} (z_t, \lambda_{t+1}) -  \hat{\nabla}_z\mathcal{L} (z_t, \lambda_{t+1}), z_{t+1} - z\rangle  + (\lambda - \lambda_{t+1})\tilde{h} (z_{t+1}) \nonumber\\
	&\qquad \qquad+ \frac{\eta_t- \mu}{2}\|z - z_t\|_2^2 - \frac{\eta_t- (\rho_f + \lambda_{t+1}\rho_h)}{2}\|z_{t+1} - z_t\|_2^2 - \frac{\eta_t}{ 2}\|z - z_{t+1}\|_2^2.\nonumber\\
	&\quad\overset{(ii)}\le  \langle {\nabla}_z\mathcal{L} (z_t, \lambda_{t+1}) -  \hat{\nabla}_z\mathcal{L} (z_t, \lambda_{t+1}), z_{t+1} - z\rangle +  (\lambda - \lambda_{t+1}) (\tilde{h} (z_{t+1}) - \hat{h} (z_{t+1}))  \nonumber \\
	&\qquad \qquad-(\lambda_{t+1} -\lambda)\xi_{t+1}+ \theta_t(\lambda_{t+1} -\lambda)\xi_{t} + \frac{\tau_t}{2}\left((\lambda -\lambda_t)^2 - (\lambda_{t+1} -\lambda_t)^2 - (\lambda - \lambda_{t+1})^2\right)\nonumber\\
	&\qquad\qquad\quad + \frac{\eta_t- \mu }{2}\|z - z_t\|_2^2 - \frac{\eta_t- (\rho_f + B\rho_h)}{2}\|z_{t+1} - z_t\|_2^2 - \frac{\eta_t}{ 2}\|z - z_{t+1}\|_2^2, \label{eq:gapboundhere}
\end{align}
where $(i)$ follows from \cref{eq:fwrtzhere} and $(ii)$ follows from \cref{eq:fwrtlambdahere} and $0\le \lambda_{t+1}\le B$.

Now we proceed to bound the term $|\tilde{h} (z_t) - \hat{h} (z_t)|$.
\begin{align}
	|\tilde{h} (z_t) - \hat{h} (z_t)| = |g(x_t, \tilde{y}^*_t) - g(x_t, \hat{y}^*_t)|\overset{(i)}\le 2L_g\|\tilde{y}^*_t - \hat{y}^*_t\|_2\overset{(ii)}\le L_gD_\mathcal{Z} \left(1 - \tfrac{\alpha}{\rho_g + 2\alpha}\right)^N,\label{eq:hgdboundhere}
\end{align}
where $(i)$ follows from \Cref{ass:smoothness}, $\mathcal{Z}$ is bounded set, and because we let $L_g:= \sup_{z}\|\nabla_z g(z)\|_2$, and $(ii)$ follows from  \Cref{lemma:pgd} and because $\|\hat{y}_0 - \tilde{y}^*(x_t)\|_2\le D_\mathcal{Z}$.

The following inequality follows immediately from \cref{eq:hgdboundhere} and the fact that $|\lambda - \lambda_{t+1}|\le B$:
\begin{align}
	(\lambda -\lambda_{t+1}) (\tilde{h} (z_t) - \hat{h} (z_t))\le |\lambda -\lambda_{t+1}||\tilde{h} (z_t) - \hat{h} (z_t)|\le L_gBD_\mathcal{Z} \left(1 - \tfrac{\alpha}{\rho_g + 2\alpha}\right)^N. \label{eq:hgdbound2here}
\end{align}
By the definitions of ${\nabla}_z\mathcal{L} (z_t, \lambda_{t+1})$ and $\hat{\nabla}_z\mathcal{L} (z_t, \lambda_{t+1})$, we have 
\begin{align}
	\|{\nabla}_z\mathcal{L} &(z_t, \lambda_{t+1}) -  \hat{\nabla}_z\mathcal{L} (z_t, \lambda_{t+1})\|_2 \nonumber \\
	&=\left\|\nabla f (z_t) + \lambda_{t+1}\nabla \tilde{h} (z_t) -\left(\nabla f(z_t) + \lambda_{t+1}\hat\nabla\tilde{h} (z_t)\right)\right\|_2\nonumber\\
	&= \lambda_{t+1}\left\| \nabla g (x_t, \tilde{y}^*_t) -\nabla g (x_t, \hat{y}_t^*)\right\|_2 \nonumber \\ &\overset{(i)}\le\lambda_{t+1}\rho_g\|\tilde{y}^*_t - \hat{y}^*_t\|_2\overset{(ii)}\le B\rho_gD_\mathcal{Z} \left(1 - \tfrac{\alpha}{\rho_g + 2\alpha}\right)^N,\label{eq:lambdagdboundhere}
\end{align}
where $(i)$ follows from \Cref{ass:smoothness} and  $(ii)$ follows from \Cref{lemma:pgd}, and because $\lambda_{t+1}\le B$ and $\|\hat{y}_0 - \tilde{y}^*(x_t)\|_2\le D_\mathcal{Z}$.

By Cauchy-Schwartz inequality and \cref{eq:lambdagdboundhere}, we have
\begin{align}
	&\langle {\nabla}_z\mathcal{L} (z_t, \lambda_{t+1}) -  \hat{\nabla}_z\mathcal{L} (z_t, \lambda_{t+1}), z_{t+1} - z\rangle\nonumber\\ 
	&\qquad\le \|\nabla_z\mathcal{L} (z_t, \lambda_{t+1}) -  \hat{\nabla}_z\mathcal{L} (z_t, \lambda_{t+1})\|_2\|z_{t+1} - z\|_2\le B\rho_gD_\mathcal{Z}^2 \left(1 - \tfrac{\alpha}{\rho_g + 2\alpha}\right)^N.\label{eq:innerproducthere} 
\end{align}
By the definition of $\xi_t$, we have 
\begin{align}
	\theta_t(\lambda_{t+1} -\lambda_t)\xi_t &= \theta_t(\lambda_{t+1} -\lambda_t)(\hat{h} (z_t) - \hat{h} (z_{t-1}))\nonumber\\
	&= \theta_t(\lambda_{t+1} -\lambda_t)(\hat{h} (z_t) - \tilde{h} (z_t) - \hat{h} (z_{t-1}) + \tilde{h} (z_{t-1}) + \tilde{h} (z_t) - \tilde{h} (z_{t-1}))\nonumber\\
	&\le\theta_t |\lambda_{t+1} -\lambda_t|\left(|\hat{h} (z_t) - \tilde{h} (z_t)|+ |\hat{h} (z_{t-1}) - \tilde{h} (z_{t-1})| + |\tilde{h} (z_t) - \tilde{h} (z_{t-1})|\right)\nonumber\\
	&\overset{(i)}\le |\lambda_{t+1} -\lambda_t|\left(2 L_gD_\mathcal{Z} \left(1 - \tfrac{\alpha}{\rho_g + 2\alpha}\right)^N + L_g\|z_t - z_{t-1}\|_2\right)\nonumber\\
	&\overset{(ii)}\le 2BL_gD_\mathcal{Z} \left(1 - \tfrac{\alpha}{\rho_g + 2\alpha}\right)^N + L_g|\lambda_{t+1} -\lambda_t|\|z_t - z_{t-1}\|_2\nonumber\\
	&\overset{(iii)}\le 2BL_gD_\mathcal{Z} \left(1 - \tfrac{\alpha}{\rho_g + 2\alpha}\right)^N + \frac{\tau_t}{2} (\lambda_{t+1} -\lambda_t)^2 + \frac{L_g}{2\tau_t}\|z_t - z_{t-1}\|_2^2,\label{eq:boundsonxihere}
\end{align}
where $(i)$ follows from \cref{eq:hgdboundhere}, and because $\theta_t\le1$, and $\tilde{h} (z)$ is $L_g$ Lipschitz continuous, $(ii)$ follows because $0\le \lambda_t, \lambda_{t+1}\le B$, and $(iii)$ follows from Young's inequality. 

Substituting \cref{eq:boundsonxihere,eq:innerproducthere,eq:hgdbound2here} into \cref{eq:gapboundhere} yields
\begin{align}
	Q(w_{t+1}, w)&\le  -(\lambda_{t+1} -\lambda)\xi_{t+1} + 
	\theta_t(\lambda_{t} -\lambda)\xi_{t} +  (\rho_gD_\mathcal{Z} + 3L_g)BD_\mathcal{Z}\left(1 - \tfrac{\alpha}{\rho_g + 2\alpha}\right)^N \nonumber\\
	&\qquad + \frac{\tau_t}{2}\left((\lambda -\lambda_t)^2 - (\lambda - \lambda_{t+1})^2\right)+  \frac{\eta_t- \mu }{2}\|z - z_t\|_2^2  - \frac{\eta_t}{2}\|z - z_{t+1}\|_2^2\nonumber\\
	&\qquad\quad  +\frac{L_g^2}{2\tau_t}\|z_t - z_{t-1}\|_2^2 -\frac{\eta_t- (\rho_f + B\rho_h)}{ 2}\|z_t - z_{t+1}\|_2^2.\label{eq:finalQre}
\end{align} 
Recall that $\gamma_t$, $\theta_t$, $\eta_t$ and $\tau_t$ are set to satisfy $\gamma_{t+1}\theta_{t+1} =\gamma_t$, $\gamma_t\tau_t \ge \gamma_{t+1}\tau_{t+1}$, $\gamma_t\eta_t \ge \gamma_{t+1}(\eta_{t+1} -\mu)$, and 
\begin{align*}
	\gamma_t(\rho_f + B\rho_h -\eta_t) + \frac{2\gamma_{t+1}L_g^2}{\tau_{t+1}} \le 0.
\end{align*}
Multiplying $\gamma_t$ on both sides of \cref{eq:finalQre} and telescoping from $t=0, 1, \ldots T-1$ yield
\begin{align}
	\sum_{t=0}^{T-1} \gamma_t Q(w_{t+1}, w) \le &-\gamma_{T-1} (\lambda_T -\lambda) \xi_T + (\rho_gD_\mathcal{Z} + 3L_g)BD_\mathcal{Z}\left(1 - \tfrac{\alpha}{\rho_g + 2\alpha}\right)^N \sum_{t=0}^{T-1} \gamma_t\nonumber\\
	& \quad+\frac{\gamma_0\tau_0}{2} (\lambda -\lambda_0)^2 + \frac{\gamma_0(\eta_0 -\mu)}{2}\|z- z_0\|_2^2 \nonumber\\ 
	&\qquad - \frac{\gamma_{T-1}(\eta_{T-1} -(\rho_f + B\rho_h))}{2}\|z -z_T\|_2^2.\nonumber
\end{align}
Dividing both sides of the above inequality by $\Gamma_T = \sum_{t=0}^{T-1}\gamma_t$, we obtain
\begin{align}
	\frac{1}{\Gamma_T}\sum_{t=0}^{T-1} \gamma_t Q(w_{t+1}, w) \le& -\frac{\gamma_{T-1} (\lambda_T -\lambda) \xi_T}{\Gamma_T} + (\rho_gD_\mathcal{Z} + 3L_g)BD_\mathcal{Z}\left(1 - \tfrac{\alpha}{\rho_g + 2\alpha}\right)^N \nonumber\\
	&\quad +\frac{\gamma_0\tau_0}{2\Gamma_T} (\lambda -\lambda_0)^2 + \frac{\gamma_0(\eta_0 -\mu)}{2\Gamma_T}\|z- z_0\|_2^2 \nonumber\\
	&\qquad - \frac{\gamma_{T-1}(\eta_{T-1} -(\rho_f + B\rho_h))}{2\Gamma_T}\|z -z_T\|_2^2.
	\label{eq:telescopingre}
\end{align}
By following the steps similar to those in \cref{eq:boundsonxihere}, we have 
\begin{align*}
	|(\lambda_{T} - \lambda) \xi_{T}|  &\le |\lambda_{T} -\lambda|\left(2 L_gD_\mathcal{Z} \left(1 - \tfrac{\alpha}{\rho_g + 2\alpha}\right)^N + L_g\|z_T - z_{T-1}\|_2\right)\\
	&\le 2L_gBD_\mathcal{Z} \left(1 - \tfrac{\alpha}{\rho_g + 2\alpha}\right)^N   + L_gBD_\mathcal{Z}.
\end{align*}
Recall the definition: $\bar w \coloneqq \tfrac{1}{\Gamma_T}\sum_{t=0}^{T-1} \gamma_tw_{t+1}$. Noting that $Q(\cdot, w)$ is a convex function and substituting the above inequality into \cref{eq:telescopingre} yield
\begin{align}
	&Q(\bar{w}, w)\nonumber\\
	&\quad\le \frac{1}{\Gamma_T}\sum_{t=0}^{T-1} \gamma_t Q(w_{t+1}, w)\nonumber\\
	&\quad\le \frac{2L_gBD_\mathcal{Z}}{\Gamma_T} \left(1 - \tfrac{\alpha}{\rho_g + 2\alpha}\right)^N + \frac{L_gBD_\mathcal{Z}}{\Gamma_T} +  (\rho_gD_\mathcal{Z} + 3L_g)BD_\mathcal{Z}\left(1 - \tfrac{\alpha}{\rho_g + 2\alpha}\right)^N\nonumber\\
	& \qquad\quad +\frac{\gamma_0\tau_0}{2\Gamma_T} (\lambda -\lambda_0)^2 + \frac{\gamma_0(\eta_0 -\mu)}{2\Gamma_T}\|z- z_0\|_2^2  - \frac{\gamma_{T-1}(\eta_{T-1} -(\rho_f + B\rho_h))}{2\Gamma_T}\|z -z_T\|_2^2.\label{eq:finalQboundsre}
\end{align}
Let $w = (z^*, 0)$. Then, we have 
\begin{align}
	Q(\bar{w}, w) = f(\bar{z}) - f(z^*) - \bar{\lambda} \tilde{h}(z^*)\overset{(i)}\ge f(\bar{z}) - f(z^*),\nonumber
\end{align}
where $(i)$ follows from the fact $\tilde{h}(z^*)\le 0$ and $\bar{\lambda} = \frac{1}{\Gamma_T}\sum_{t=0}^{T-1} \gamma_t\lambda_{t+1} \ge 0$.

Substituting the above inequality into \cref{eq:finalQboundsre} yields
\begin{align}
	f(\bar{z}) - f( z^*)&\le  \frac{2L_gBD_\mathcal{Z}}{\Gamma_T}\left(1 - \tfrac{\alpha}{\rho_g + 2\alpha}\right)^N + \frac{L_gBD_\mathcal{Z}}{\Gamma_T}\nonumber\\
	&\qquad + (\rho_gD_\mathcal{Z} + 3L_g)BD_\mathcal{Z}\left(1 - \tfrac{\alpha}{\rho_g + 2\alpha}\right)^N + \frac{\gamma_0(\eta_0 -\mu)\|{z}^*- z_0\|_2^2}{2\Gamma_T}.\label{eq:optimalitygapre}
\end{align}

Recall that $({z}^*, {\lambda}^*)$ is a Nash equilibrium of $\mathcal{L}(z, \lambda)$ and it satisfies  ${\lambda}^*\tilde{h}({z}^*) = 0$. Then we have
\begin{align}
	\mathcal{L}(\bar z, \lambda^*) \ge \mathcal{L}({z}^*, \lambda^*) \overset{\mbox{by def.}}\Longleftrightarrow f(\bar z) + \lambda^* \tilde{h}(\bar{z}) - f({z}^*)\ge 0. \label{eq:final1re}
\end{align}
If $\tilde{h}(\bar{z})\le 0$, the constraint violation $[\tilde{h}(\bar{z})]_+ = 0$, which satisfies the statement in the theorem. If $\tilde{h}(\bar z)>0$, let $w= ({z}^*, \lambda^*+1)$. Then, we have
\begin{align}
	Q(\bar w, w) = f(\bar{z}) +  (\lambda^*+1)\tilde{h}(\bar{z}) - f(z^*) - \bar\lambda \tilde{h}(z^*)\overset{(i)}\le f(\bar z) + (\lambda^* + 1)\tilde{h}(\bar z) - f(z^*),\label{eq:final2re}
\end{align}
where $(i)$ follows from the facts $\tilde{h}(z^*)\le 0$ and $\bar \lambda\ge 0$.

 \Cref{eq:final1re,eq:final2re,eq:finalQboundsre} and the condition $\tilde{h}(\bar{z})>0$ together yield,
\begin{align}
	[\tilde{h}(\bar{z})]_+ =& \tilde{h}(\bar{z}) = Q(\bar w, w) - (f(\bar w) + \lambda^* \tilde{h}(\bar{z}) - f({z}^*))\le  Q(\bar{w}, w)\nonumber\\
	\le & \frac{2L_gBD_\mathcal{Z}}{\Gamma_T} \left(1 - \tfrac{\alpha}{\rho_g + 2\alpha}\right)^N + \frac{L_gBD_\mathcal{Z}}{\Gamma_T} +  (\rho_gD_\mathcal{Z} + 3L_g)BD_\mathcal{Z}\left(1 - \tfrac{\alpha}{\rho_g + 2\alpha}\right)^N\nonumber\\
	&  +\frac{\gamma_0\tau_0}{2\Gamma_T} (\lambda^*+1)^2 + \frac{\gamma_0(\eta_0 -\mu)}{2\Gamma_T}\|z^*\|_2^2.\label{eq:constraintviolationre}
\end{align}

Finally, taking $w^*=({z}^*, {\lambda}^*)$ in \cref{eq:finalQboundsre}, noticing the fact $Q(w, w^*)\ge0$ for all $w$, and rearranging the terms, we have 
\begin{align}
	\|\bar{z} - z^*\|_2^2 &\le \frac{1}{\gamma_{T-1}(\eta_{T-1} -(\rho_f + B\rho_h))}\left(4L_gBD_\mathcal{Z}\left(1 - \tfrac{\alpha}{\rho_g + 2\alpha}\right)^N + 2L_gBD_\mathcal{Z} \right)\nonumber\\
	& \qquad +\frac{1}{\gamma_{T-1}(\eta_{T-1} -(\rho_f + B\rho_h))}\left({\gamma_0\tau_0} ({\lambda}^* -\lambda_0)^2 + {\gamma_0(\eta_0 -\mu)}\|{z}^*- z_0\|_2^2\right)\nonumber\\
	&\quad\qquad+  \frac{\Gamma_T}{\gamma_{T-1}(\eta_{T-1} -(\rho_f + B\rho_h))}(\rho_gD_\mathcal{Z} + 3L_g)BD_\mathcal{Z}\left(1 - \tfrac{\alpha}{\rho_g + 2\alpha}\right)^N.\label{eq:distanceboundre}
\end{align}
Moreover, using the fact that $\Gamma_T\ge \tfrac{T^2}{2}$ and $\Gamma_T\ge 2$, \cref{eq:optimalitygapre} yields
\begin{align*} 
	f(\bar{z}) - f( z^*)&\le   \frac{2L_gBD_\mathcal{Z}}{T^2} + \frac{\gamma_0(\eta_0 -\mu)\|{z}^*- z_0\|_2^2}{T^2}+ (\rho_gD_\mathcal{Z} + 4L_g)BD_\mathcal{Z} \left(1 - \tfrac{\alpha}{\rho_g + 2\alpha}\right)^N .
\end{align*}
\Cref{eq:constraintviolationre} together with the facts that $\Gamma_T\ge 2$,$\Gamma_T\ge T^2/2$, $\|z^*\|_2\le D_\mathcal{Z}$, and $\lambda^*+1\le B$, implies
\begin{align*}
    	[\tilde{h}(\bar{z})]_+ \le \frac{2L_gBD_\mathcal{Z}+ \gamma_0\tau_0 B^2+\gamma_0(\eta_0 -\mu){D}_\mathcal{Z}^2}{T^2} +  (\rho_gD_\mathcal{Z} + 4L_g)BD_\mathcal{Z}\left(1 - \tfrac{\alpha}{\rho_g + 2\alpha}\right)^N.
\end{align*}
Using the fact that $\gamma_{T-1}(\eta_{T-1} - \rho_f -B\rho_h)\ge \tfrac{\mu T^2}{2}$ and $\Gamma_T\le \mu (T+t_0 + 2)^2$, \cref{eq:distanceboundre} yields
\begin{align*}
    \|\bar{z} - z^*\|_2^2 \le \frac{2{\gamma_0\tau_0} B^2 + 2{\gamma_0(\eta_0 -\mu)}D_\mathcal{Z}^2}{\mu T^2} +  \frac{2(T+t_0+1)^2}{T^2}(\rho_gD_\mathcal{Z} + 4L_g)BD_\mathcal{Z}\left(1 - \tfrac{\alpha}{\rho_g + 2\alpha}\right)^N.
\end{align*}
\end{proof}

\section{Proof of \Cref{lemma:lipschizh} in \Cref{sec:proximalpdbo}}

Recall that we have already shown that $\|\nabla^2 \tilde{g}^*(z)\|_2\le \rho_g + \tfrac{\rho_g^2}{\alpha}$ in \cref{eq:nablagstar2} for all $z\in\mathcal{Z}$. Then, the following inequality holds for all $x, x^\prime\in\mathcal{X}$
\begin{align}
     \tilde g^*(x^\prime)\le \tilde g^*(x) + \langle\nabla_x \tilde g^*(x), x^\prime - x\rangle + \tfrac{\rho_g + \rho_g^2/\alpha}{2} \|x^\prime -x \|_2^2.\label{eq:middle62}
\end{align}
Recall the inequality in \Cref{ass:smoothness}:
\begin{align}
     g(x^\prime, y^\prime)\ge g(x, y) + \langle\nabla_x g(x, y), x^\prime-x\rangle + \langle \nabla_y g(x, y), y^\prime -y\rangle- \tfrac{\rho_g}{2}\|x - x^\prime\|_2^2.\label{eq:middle63}
\end{align}
Consider $\tilde{h}_k(z)$. We have $\nabla_z \tilde{h}_k(z) = \nabla_z g(z) - (\nabla_x \tilde g^*(x); \mathbf{0}_d) + (2\rho_g + \tfrac{\rho_g^2}{\alpha})(x - \tilde{x}_{k-1};\mathbf{0}_d)$. Then, we have 
\begin{align}
    &\tilde h_k(z) + \langle\nabla_z \tilde{h}_k(z), z^\prime - z\rangle\nonumber\\
    &\quad= g(x, y) - \tilde{g}^*(x) + \tfrac{(2\rho_g + \tfrac{\rho_g^2}{\alpha})}{2}\|x - \tilde{x}_{k-1}\|_2^2- \delta + \langle \nabla_x g(x, y), x^\prime - x\rangle -\langle \nabla_x \tilde g^*(x), x^\prime - x\rangle\nonumber\\
    &\qquad\quad+\langle \nabla_y \tilde g(x, y), y^\prime - y\rangle + (2\rho_g + \tfrac{\rho_g^2}{\alpha})\langle x - \tilde{x}_{k-1},  x^\prime - x\rangle .\label{eq:middle64}
\end{align}
Substituting \cref{eq:middle62,eq:middle63} into \cref{eq:middle64}, we obtain
\begin{align*} 
    \tilde h_k(z) + \langle\nabla_z \tilde{h}_k(z), z^\prime - z\rangle \le g(x^\prime, y^\prime) -\tilde g(x^\prime) + \tfrac{2\rho_g + \rho_g^2/\alpha}{2}\|x^\prime - \tilde x_{k-1}\|_2^2 - \delta= \tilde h_k(z^\prime).
\end{align*}
The above inequality is the necessary and sufficient condition for convexity, which completes the proof.

\section{Proof of \Cref{thm:maintheorem}}\label{sec:proofofproximalpdo}

We first formally restate the theorem with the full details.
\begin{theorem}[Formal Statement of \Cref{thm:maintheorem}]
	Suppose that \Cref{ass:smoothness} holds.  Consider \Cref{alg:proximalmethod}. Let the hyperparameters $B>0$ be a large enough constant, $\gamma_t = t+t_0+1$, $\eta_t = \frac{\rho_f(t+t_0+1)}{2}$, $\tau_t = \frac{4L_g^2}{\rho_f t}$, $\theta_t = \frac{t+t_0}{t+t_0+1}$, where $t_0 = \frac{6\rho_f+ 4B\rho_h}{\rho_f}$, $\rho_h$ is specified in \Cref{lemma:lipschitzh}, and $L_g= \sup_{z\in\mathcal{Z}}\|\nabla g(z)\|_2$. Set $B = \tfrac{D_f+\rho_fD_\mathcal{Z}^2}{\delta}+1$, where $D_f = \sup_{z, z^\prime\in\mathcal{Z}}|f(z) - f(z^\prime)|$. Then, the output $\tilde{z}_{\hat{k}}$ of \Cref{alg:proximalmethod} with a randomly chosen index $\hat k$ is a stochastic $\epsilon$-KKT point of \cref{eq:reformulation}, where $\epsilon$ is given by
	$\epsilon=\mathcal{O}\left(\tfrac{1}{K}\right)+\mathcal{O}\left(\tfrac{1}{T^2}\right)+ \mathcal{O}\left(e^{-N}\right)$.
\end{theorem}

Central to the proof of \Cref{thm:maintheorem}, we first prove the uniform bound of the optimal dual variables as stated in the following lemma.
\begin{lemma}\label{lemma:uniformboundre}
For each subproblem $(\mbox{P}_k)$, there exists a unique global optimizer ${z}_k^*$ and optimal dual variable ${\lambda}_k^*$ such that ${\lambda}_k^* \le \bar{B} \coloneqq  (D_f + \rho_fD_\mathcal{Z}^2)/\delta$, where $D_f\coloneqq\sup_{z, z^\prime\in\mathcal{Z}} |f(z) - f(z^\prime)|$.
\end{lemma}
\begin{proof}[Proof of \Cref{lemma:uniformboundre}]
For each subproblem $(\mbox{P}_k)$, let $\bar{z}_{k-1} =(\tilde{x}_{k-1}, \bar y_{k-1})$ with  $\bar{y}_{k-1} =\argmin_{\mathcal{Y}} g(\tilde{x}_{k-1}, y) + \frac{\alpha\|y\|_2^2}{2}$. Then, we have $\tilde{h}_k(\bar{z}_{k-1}) = -\delta$, which is a strictly feasible point. 

Define the function $d_k(\lambda) = \min_{z\in\mathcal{Z}} \mathcal{L}_k(z, \lambda)$. Then, for any $\lambda$ and $z\in\mathcal{Z}$, we have
\begin{align}\label{eq:dualfunctionboundre}
	d_k(\lambda) \le  f_k(\bar{z}_{k-1}) + \lambda\tilde h_k(\bar{z}_{k-1}) = f_k(\bar{z}_{k-1})  -\delta\lambda.
\end{align}
Moreover, it is known that constrained optimization with strongly convex objective and strongly convex constraints has no duality gap. Combining this with the fact that the strictly feasible point exists, we conclude that the optimal dual variable exists in $\mathbb{R}_+$. Taking $\lambda=\lambda^*_k$ in \cref{eq:dualfunctionboundre} and using the fact that $|d_k(\lambda^*)- f_k(\bar{z}_{k-1})| = | f_{k}(z_k^*)- f_k(\bar{z}_{k-1})|\le D_f + \rho_fD_\mathcal{Z}^2$, we complete the proof.
\end{proof}


To proceed the proof of \Cref{thm:maintheorem}, the function $f_k(z)$ in the subproblem $(\mathrm{P}_k)$ is a $\mu=\rho_f$ strongly convex and $3\rho_f$ Lipschitz continuous function, and $h_k(z)$ is a convex function and $2\rho_h$ Lipschitz continuous function.  Thus, as we state in the theorem, let $\gamma_t$, $\tau_t$, $\theta_t$ and $\eta_t$ be the same as those in \Cref{thm:pdboconvergencere} with $\mu=\rho_f$, $\rho_f$ being replaced by $3\rho_f$, and $\rho_h$ being replaced by $2\rho_h$. We then follow steps similar to those in the proof of \Cref{thm:pdboconvergencere}. In particular, note that \Cref{lemma:uniformboundre} indicates that for each subproblem $(\mathrm{P}_k)$, the optimal dual variable exists and is bounded by $\bar{B} = (D_f+\rho_fD_\mathcal{Z}^2)/\delta$. Thus, setting $\Lambda= [0, B]$ with $B= \bar{B} + 1$ ensures both $\lambda^*_k$ and $\lambda^*_k+1$ to be inside the set $\Lambda$, which further ensures that we can follow steps similar to  \cref{eq:final2re,eq:distanceboundre}. We then obtain for all $k\in\mathbb{N}$, the following bounds hold:
\begin{align*} 
	f_k(\tilde{z_k}) - f_k(z_k^*)&\le  \mathcal{O}\left(\frac{1}{T^2}\right)+ \mathcal{O}\big(e^{-N}\big),
\end{align*}
\begin{align*}
    	[\tilde{h}_k(\tilde{z}_k)]_+ \le  \mathcal{O}\left(\frac{1}{T^2}\right)+ \mathcal{O}\big(e^{-N}\big),
\end{align*}
and
\begin{align*}
    \|\tilde{z}_k - z_k^*\|_2^2 \le  \mathcal{O}\left(\frac{1}{T^2}\right)+ \mathcal{O}\big(e^{-N}\big).
\end{align*}

With the above convergence rate bounds, we apply the following result on the convergence of the nonconvex constrained problem.
\begin{lemma}[Theorem 3.17 \cite{boob2019stochastic}]\label{thm:outerloop}
Suppose \Cref{ass:smoothness} hold. Denote the global optimizer of $(\mbox{P}_k)$ as $z_k^*$.  Suppose the optimal dual variable has an upper bound $\bar B$, each subproblem is solved to $\Delta$-accuracy, i.e., the optimality gap $f_k(\tilde{z}_k) - f_k( z_k^*)\le \Delta$, constraint violation $[\tilde{h}_k(\tilde{z}_k)]_+\le \Delta$, and distance to the solution $\|\tilde{z}_k - {z}_k^*\|_2^2\le \Delta$. Then, the final output with a randomly chosen index is an stochastic $\epsilon$-KKT point of \cref{eq:reform3}, with $\epsilon=\frac{D_f+\rho_fD_\mathcal{Z}^2 + \bar{B}V_0}{K} + \tfrac{16(\bar{B}+1)}{\rho_f} \Delta$, where $D_f=\sup_{z, z^\prime\in\mathcal{Z}} |f(z) -f(z^\prime)|$, $V_0= \max\{\tilde{h}(\tilde{z}_0), 0\}$. 
\end{lemma}

Applying \Cref{thm:outerloop} with $\Delta= \mathcal{O}(\tfrac{1}{T^2}) + \mathcal{O}(e^{-N})$, we conclude that $\tilde{z}_{\hat{k}}$ is an $\epsilon$-KKT point of \cref{eq:reform3} where $\epsilon$ is specified below as
\begin{equation*}
    \epsilon =\mathcal{O}\left(\frac{1}{K}\right) +  \mathcal{O}\left(\frac{1}{T^2}\right)+ \mathcal{O}\big(e^{-N}\big).
\end{equation*}
This completes the proof.

\end{document}